\documentclass[ECP]{ejpecp}

\usepackage{notation}

\usepackage{todonotes}

\SHORTTITLE{Perturbed KL Bounds for DPs}

\TITLE{Perturbed Kullback–Leibler Deviation Bounds for Dirichlet Processes
    } 

\AUTHORS{%
Pierre Perrault\footnote{IDEMIA.  \EMAIL{pierre.perrault@outlook.com}}
}

\KEYWORDS{Dirichlet Process; Kullback–Leibler divergence; superadditivity; Fekete’s lemma; deviation bounds} 

\AMSSUBJ{60E15; 62E17} 

\SUBMITTED{...} 
\ACCEPTED{...} 

\VOLUME{0}
\YEAR{2025}
\PAPERNUM{0}
\DOI{10.1214/YY-TN}

\ABSTRACT{
We present new and improved non-asymptotic deviation bounds for Dirichlet processes (DPs), formulated using the Kullback–Leibler (KL) divergence, which is known for its optimal characterization of the asymptotic behavior of DPs. Our method involves incorporating a controlled perturbation within the KL bound, effectively shifting the base distribution of the DP in the upper bound. Our proofs rely on two independent approaches. In the first, we use superadditivity techniques to convert asymptotic bounds into non-asymptotic ones via Fekete's lemma. In the second, we carefully reduce the problem to the Beta distribution case. Some of our results extend similar inequalities derived for the Beta distribution, as presented in \cite{henzi2023some}.
}

\begin{document}
\section{Introduction}
\label{sec:introduction}
In probability theory and statistics, the Beta distribution \(\text{Beta}(a, b),\) with parameters \( a, b > 0 \), is a ubiquitous continuous probability distribution defined on the interval $[0, 1]$. Its tail probability is given by
\[ 
\PPs{X\sim \text{Beta}(a, b)}{X \geq u} \triangleq \frac{\Gamma(a + b)}{\Gamma(a) \Gamma(b)}\int_u^1 {x^{a-1} (1 - x)^{b-1}}{} \, dx, \quad 0 \leq u \leq 1,
\]
where \( \Gamma(z) \triangleq \int_0^\infty t^{z-1}e^{-t}dt\) is  the gamma function. Computing \(\PPs{X\sim \text{Beta}(a, b)}{X \geq u}\) is generally intractable~\cite{dutka1981incomplete}; yet accurately estimating it is crucial in various fields, including large deviation theory \cite{zhang2020non}, Bayesian nonparametrics \cite{castillo2017polya}, adaptive Bayesian inference \cite{elder2016bayesian} and random matrix theory \cite{frankl1990some}. This has led to the development of several closed-form exponential upper bounds, with some of the most notable summarized in Table~\ref{tab:bounds}. 
Among these, Kullback–Leibler (KL)-type bounds are particularly noteworthy due to their asymptotic optimality:  
\[\log\PPs{X\sim \text{Beta}(nx, n(1-x))}{X \geq u} \sim_{n\to\infty} -n\kl{x}{u},\quad \text{for }u,x\in (0,1),\text{ see e.g., \cite{lynch1987large}}.\]
In particular, the perturbed KL bound from \cite{henzi2023some} is often considered the most effective, as it is tighter than the Hoeffding-type inequality \cite[Corollary 5]{henzi2023some} and the Bernstein-type inequality \cite[Lemma 6]{henzi2023some}.
This improvement arises from introducing a controlled perturbation \( \eta > 0 \) into the KL bound, enabling slight adjustments to the base mean \( a/(a+b) \). Specifically, the perturbed expression \((a-\eta)/(a+b-\eta)\) tightens the KL divergence by effectively lowering the mean and shifting it slightly further from the reference point \( u \).
\begin{table}[t]
\centering
\begin{tabular}{|c|c|c|}
\hline
Bound on $\log\PPs{X\sim \text{Beta}(a, b)}{X \geq u}$ & Validity range & Type
\\
  \hline
  \transparent{0}
\fbox{\transparent{1}${-2(a+b+1)\pa{u - \frac{a}{a+b}}^2}$} & \multirow{5}{*}{$\frac{a}{a+b}\leq u\leq 1$} & Hoeffding \cite{marchal2017sub} 
  \\
  \cline{1-1}\cline{3-3}
  \transparent{0}\fbox{\transparent{1}
  \(-\frac{\pa{(a+b)u - {a}}^2}{\frac{2ab}{a+b+1}+\frac{4 \pa{(a+b)u - {a}}\pa{b-a}^+}{3(a+b+2)}} \)} && Bernstein \cite{skorski2023bernstein} \\
  \cline{1-1}\cline{3-3}\transparent{0}\fbox{\transparent{1}
  $-(a+b)\kl{\frac{a}{a+b}}{u}$}& & Kullback–Leibler (KL) \cite{dumbgen1998new} \\
  \hline   \transparent{0}\fbox{\transparent{1}
  \(\begin{array}{cc}
  -(a+b-\eta)\kl{\frac{a-\eta}{a+b-\eta}}{u}, \\
  \text{for }a\geq 1, b > 0, \eta= 1+\frac{a-1}{b+1}. 
    \end{array}\)} & {$\frac{a-\eta}{a+b-\eta}\leq u\leq 1$} & {perturbed KL \cite{henzi2023some}}\\
  \hline
\end{tabular}
\caption{Table of well-known exponential upper bounds for the Beta distribution, where $\kl{p}{q}\triangleq p\log(p/q)+(1-p)\log((1-p)/(1-q))\in [0,\infty]$ for $p,q\in [0,1].$}
\label{tab:bounds}
\end{table} 

Similarly, the Dirichlet distribution and its infinite-dimensional extension, the Dirichlet process (DP), introduced by Ferguson in the early 1970s \cite{ferguson1973bayesian}, frequently rely on tail probability bounds in various applications, including multi-armed bandits \cite{belomestny2023sharp}, reinforcement learning \cite{tiapkin2022dirichlet}, and Bayesian bootstrap methods \cite{rubin1981bayesian}. One of the most commonly used bounds is a simple extension of the standard KL bound from \cite{dumbgen1998new}, transitioning from the Beta distribution to DPs by replacing the binary KL with the multivariate KL; see \eqref{rel:chernoff_(1)_lemma1} for its exact formulation. This raises a natural question: Can perturbed KL-type tail probability bounds be extended to Dirichlet distributions and DPs?
In this work, we provide an affirmative answer by introducing new perturbed KL bounds for DPs, which not only refine existing DP results but also unify and extend known bounds for the Beta distribution. Before proceeding, we first introduce the relevant notation.

\subsection{Notation and preliminaries}
For a probability distribution \( p \), we use \( \mathbb{E}_p \) and \( \mathbb{P}_p \) to denote the expectation and probability with respect to \( p \), respectively. Additionally, as introduced earlier, we use the shorthand notations \( \mathbb{E}_{\xi \sim p} \) and \( \mathbb{P}_{\xi \sim p} \) in place of \( \mathbb{E}_p \) and \( \mathbb{P}_p \).
To recall, the Kullback–Leibler (KL) divergence between two probability distributions \( \mu \) and \( \nu \) is defined as:
\[
\KL{\mu}{\nu} \triangleq
\begin{cases}
\mathbb{E}_{\mu} \left[ \log \left( \frac{d\mu}{d\nu} \right) \right] & \text{if } \mu \ll \nu, \\
\infty & \text{otherwise}.
\end{cases}
\]
We consider a compact metric space \( \Omega \) equipped with its Borel \( \sigma \)-algebra \( \mathcal{B}(\Omega) \). Let \( \mathcal{M}(\Omega) \) (resp. \( \mathcal{M}_1(\Omega) \)) denote the space of finite non-negative (resp. probability) measures on \( \Omega \), and let \( \mathcal{B}(\mathcal{M}_1(\Omega)) \) denote the Borel \( \sigma \)-algebra induced by the weak topology on \( \mathcal{M}_1(\Omega) \). The set of continuous functions \( f: \Omega \to \mathbb{R} \) is denoted by \( \mathcal{C}(\Omega) \).
We consider a Dirichlet Process (DP) on \( \Omega \), characterized by a scale parameter \( \alpha > 0 \) and a base distribution \( \nu_0 \in \mathcal{M}_1(\Omega) \), whose measure's law is denoted as \( \text{DP}(\alpha \nu_0) \), and whose realization \( X \) is a random probability measure on the space \( \Omega \). We assume, without loss of generality, that \( \nu_0 \) is supported on \( \Omega \).
The original definition of the DP says that for any finite measurable partition $A_1\sqcup\dots\sqcup A_k=\Omega, \pa{X(A_1),\dots,X(A_k)}\sim\text{Dir}\pa{\alpha\nu_0(A_1),\dots,\alpha\nu_0(A_k)}$, where $\text{Dir}(a_1,\dots,a_k)$ denotes the Dirichlet distribution with
parameters $a_1,\dots,a_k$.
Another important representation of the DP is related to the characterization of the Gamma distribution by \cite{lukacs1955characterization}, and is expressed as \( X = G/G(\Omega) \), where \( G \sim \mathcal{G}(\alpha \nu_0) \) is the standard Gamma process with shape parameter \( \alpha \nu_0 \), i.e.,
$G(A) \sim \text{Gamma}(\alpha\nu_0(A), 1)$ for any measurable set $A \subset \Omega$ \cite{casella2005christian}. The key property here is that \( X = G/G(\Omega) \) is stochastically independent of \( G(\Omega) \). Since independent random variables are simpler to work with, this reparametrization proves useful for deriving properties of the DP. 

For a probability distribution \( \nu \in \mathcal{M}_1(\Omega) \), one of the central information-theoretic measures we are considering is defined as an infimum of KL divergences: for some real-valued continuous function \( f \in \mathcal{C}(\Omega) \) and some \( u \in \mathbb{R} \), we define
\[
\mathcal{K}_{\inf}(\nu, u, f) \triangleq \inf_{\mu \in \mathcal{M}_1(\Omega), \, \mathbb{E}_{\mu}[f] \geq u} \KL{\nu}{\mu},
\]
where by convention, the infimum of the empty set is defined to be \( \infty \). When $\Omega=[d]=\sset{1,\dots,d}$ is finite, we can also use the notation $\Kinf\pa{\pa{p_i}_{i\in [d]},u,\pa{f_i}_{i\in [d]}}$, for $f\in \R^d$ and $p\in \cM_1\pa{[d]}$.
Like the KL divergence, $\Kinf$ admits a variational formula, which we state in Lemma~\ref{lem:Variational formula for Kinf}.
\begin{lemma}[Variational formula for $\Kinf$]
    For all $\nu\in \cM_1(\Omega)$, $f\in \cC(\Omega)$, $u\in [f_{\min},f_{\max})$, where $f_{\max}\triangleq \max_{x\in \Omega} f(x)$, $f_{\min}\triangleq \min_{x\in \Omega} f(x)$, we have
    \[\Kinf(\nu, u, f) = \max_{\lambda\in [0,1/(f_{\max}-u)]} \EEs{\nu}{\log\pa{1-\lambda \pa{f-u}}}.\]
    Moreover, if $\lambda^*$ is the value at which the above maximum is reached, then
 \[\EEs{\nu}{1/\pa{1-\lambda^*(f-u)}}\leq 1.\]
 In particular, $\nu\pa{f^{-1}\pa{\sset{f_{\max}}}}=0$ in the case $\lambda^*=1/(f_{\max}-u)$.
 \label{lem:Variational formula for Kinf}
\end{lemma}
\noindent
This formula appears in \cite{honda2015nonasymptotic,garivier2022klucbswitch} and is an essential tool for deriving the deviation and concentration results involving $\Kinf$, as we will see in the next subsection. 
$\mathcal{K}_{\inf}$ can be interpreted as a distance from the measure $\nu$ to a set of all measures $\mu$ satisfying the constraint, where the distance is measured by the KL-divergence.
The measure $\mu$ solving this optimization problem is called moment projection ($M$-projection) or reversed information projection ($rI$-projection), see \cite{csiszar2003information,bishop2006pattern,murphy2022probabilistic}.
Since the KL-divergence is not symmetric, this is different from the more common information projection ($I$-projection), 
\(
\inf_{\mu\in S} \KL{\mu}{\nu},
\)
 appearing, for example, in Sanov-type deviation bounds \cite{sanov1961probability}.
The $I$-projections have an excellent geometric interpretation because the KL can be viewed as a Bregman divergence. The $M$-projections are not Bregman divergences and lack geometric interpretation. However, they are deeply connected to the maximum likelihood estimation when the measure $\nu$ is the empirical measure of a sample \cite[Lemma 3.1]{csiszar2004information}. Additionally, within a Bayesian framework, $M$-projections naturally arise as a rate function in the context of large deviation principles (LDP) \cite{ganesh1999inverse}, which we recall in the following paragraph. For example, we will see that $\Kinf$ is used to express the standard LDP result for DP (see Theorem~\ref{thm:gen_ldp_DP}). $M$-projections also naturally appear in lower (and sometimes upper) bounds for multi-armed bandits \cite{Lai1985asymptotically,burnetas1996optimal}. 

Finally, we briefly review the concept of large deviation principles (LDP), which plays an important role in this paper.

\begin{definition}[Rate function]
A function $I$ defined on $\cM_1\pa{\Omega}$ is a rate function if it is lower semicontinuous with values in $[0,\infty]$ (such that all level set $\sset{x, I(x)\leq t},$ for $t\in [0,\infty)$, is closed). A rate function is good if the level sets are compact.  
The effective domain of $I$ is $D_I\triangleq \sset{x, I(x)<\infty}.$
\end{definition}

\begin{definition}[Large deviation principle: LDP]\label{def:LDP}
    A sequence of probability measures $(\mu_n)$ satisfies an LDP with speed $n$ (we can avoid explicitly stating the speed if it is clear from the context) and rate function $I$ if:
    \begin{align*}
        (i)~~~~&\mbox{For all closed set $F$, $\limsup_{n}\frac{1}{n}\log\mu_n\pa{F} \leq -\inf_{x\in F} I(x)$,}\\
        (ii)~~~~&\mbox{For all open set $G$,
        $\liminf_{n}\frac{1}{n}\log\mu_n\pa{G} \geq -\inf_{x\in G} I(x)$.}
    \end{align*}
\end{definition}

\subsection{Existing KL-concentration results for DPs}

Before presenting our results, we briefly review some existing DP concentration results involving the KL. It turns out that non-asymptotic concentration bounds for DPs are limited. They are mainly based on the above Gamma representation, combined with the moment generating function (MGF) formula of the Gamma process \cite{van1997gamma, vershik2001remarks}: for each continuous function $f\leq 1$ on $\Omega$, so that $\nu_0\pa{f^{-1}\pa{\sset{1}}}=0$,
\begin{align}\textstyle\EEs{G\sim\cG(\alpha\nu_0)}{\exp\pa{\int fdG}} = \exp\pa{-\alpha \EEs{\nu_0}{\log\pa{1-f}}} 
.\label{rel:MGFGP}\end{align}
Indeed, as a simple example, with the same assumptions and notations as in Lemma~\ref{lem:Variational formula for Kinf}, one can consecutively use the Gamma representation $X=G/G(\Omega)$, Chernoff bound, equality~\eqref{rel:MGFGP} and Lemma~\ref{lem:Variational formula for Kinf} to get, for $f\in\cC(\Omega)$ and $u\in [f_{\min},f_{\max})$,
\begin{align}\nonumber
    \PPs{X\sim \DP\pa{\alpha \nu_0}}{\EEs{X}{f} \geq u} 
    &\leq \textstyle\EEs{G\sim\cG(\alpha\nu_0)}{\exp\pa{\lambda^*\int\pa{f - u}dG}}
    \\&= e^{-\alpha \EEs{\nu_0}{\log\pa{1-\lambda^*\pa{f - u}}}} = e^{-\alpha \Kinf(\nu_0, u, f)}.\label{rel:chernoff_(1)_lemma1}
\end{align}
This directly extends the standard KL bound for the Beta distribution from \cite{dumbgen1998new}, which we recover 
by considering a finite $\Omega$ and setting \( f(x) = \mathbb{I}\{x\in A\} \), for some $A\subset \Omega$: the minimizer in \( \Kinf \) is then given by   
\(
\mu = u \frac{\nu_0(\cdot \cap A)}{\nu_0(A)} + (1 - u) \frac{\nu_0(\cdot \setminus A)}{\nu_0(\Omega \setminus A)},
\)  
which leads to  
\(
\Kinf(\nu_0,u,f) = \kl{\nu_0(A)}{u}.
\)

On the other hand, Gaussian approximations of the Dirichlet distribution have been investigated to obtain probabilistic bounds. In particular, \cite{marchal2017sub} established its subgaussianity, though their bound relies on a proxy variance dictated by the largest Dirichlet parameter, leading to a rather coarse estimate. A more refined Gaussian-like approximation was introduced in \cite{belomestny2023sharp}, formulated as
$\PPs{X\sim \cN(0,1)}{X\geq \sqrt{2\alpha \Kinf(\nu_0,u,f)}}$, which exhibits a KL-like structure. 
Although related, our work differs by providing a direct KL bound rather than one expressed through a Gaussian variable.

On the asymptotic side, the literature on LDPs for DPs is 
well-established \cite{doss1982tails, lynch1987large,ganesh2000large,feng2007large,feng2023hierarchical} and it is known that the probability that $X\sim \DP\pa{\alpha\nu_0}$ deviates to another distribution $\nu$ decreases exponentially (with respect to $\alpha$) at a rate given by $\KL{\nu_0}{\nu}$. More precisely, the standard LDP result for DP is as follows.
\begin{theorem}[see \cite{ganesh2000large} for a similar statement regarding the second half of the result]\label{thm:gen_ldp_DP}
Let $(\mu_n)$ be a sequence of measures in $\cM(\Omega)$, 
 such that\[\frac{\mu_n}{n}\to\mu\in \cM_1\pa{\Omega},~~\text{weakly.}\]
If \({\log\pa{\min\pa{\mu_n(A),1}}}=o\pa{n}\) for each $\mu$-continuity set with non-empty interior $A\subset \Omega$, then $\pa{\DP\pa{\mu_n}}_n$ satisfies a LDP with speed $n$ and rate function $I(\nu) = \KL{\mu}{\nu}$.
In particular,
$\pa{\DP\pa{\alpha\nu_0}}_\alpha$ satisfies a LDP with speed $\alpha$ and rate function $I(\nu) = \KL{\nu_0}{\nu},$ i.e.,
for all $B\in \cB\pa{\cM_1(\Omega)}$, if $B^{\mathrm{o}}$ (resp. $\bar B$) denotes the interior of $B$ (resp. the closure),
\begin{align*}-\inf_{\nu\in B^{\mathrm{o}}} I(\nu)\leq \liminf_{\alpha}\frac{1}{\alpha}\log\PPs{\DP\pa{\alpha\nu_0}}{B}\leq \limsup_{\alpha}\frac{1}{\alpha}\log\PPs{\DP\pa{\alpha\nu_0}}{B} \leq -\inf_{\nu\in \bar B} I(\nu).\end{align*}
\end{theorem}
\noindent
We do not claim this result as original, as it follows from a standard approach. However, for the sake of completeness and because we are not aware of it being explicitly stated elsewhere, we provide the proof in Appendix~\ref{app:ldp}. While our work does not focus on asymptotic results, this theorem will be instrumental in establishing the desired non-asymptotic probabilistic bounds. In fact, the LDP rate function often appears in non-asymptotic concentration inequalities. For instance, consider a process of real-valued i.i.d. random variables $(Y_i)$. Then, it's easy to see that the sequence $\log\PP{{\sum_{i=1}^n Y_i}/n \geq x}$ is superadditive w.r.t. $n$. Therefore, according to the superadditive lemma due to Fekete \cite{fekete1923verteilung}, we can conclude that for all $n\geq 1$,
\begin{align*}\textstyle\frac{1}{n}\log\PP{\frac{1}{n}{\sum_{i=1}^n Y_i} \geq x} &\leq \textstyle{\sup_{m\geq 1}\frac{1}{m}\log\PP{\frac{1}{m}{\sum_{i=1}^m Y_i} \geq x}} \\&= \textstyle{\lim_{m\to\infty}\frac{1}{m}\log\PP{\frac{1}{m}{\sum_{i=1}^m Y_i} \geq x}},\end{align*}
which is minus the corresponding LDP rate function. 

In this paper, one of our goals is to apply a similar superadditivity approach for DPs.\footnote{Notably, a similar approach was employed in \cite{perrault2024new}, but with a focus on bounding the moment-generating function, whereas our goal is to directly bound the probability.} More precisely, to showcase our method, we can re-prove \eqref{rel:chernoff_(1)_lemma1}, this time using superadditivity: We first prove that the function $h:\alpha\in \R_+^* \mapsto\log\PPs{X\sim \DP\pa{\alpha \nu_0}}{\EEs{X}{f} \geq u}$
is superadditive:
\begin{align}\nonumber
\exp\pa{h\pa{\alpha}+h\pa{\beta}}
&=\nonumber
\PPs{(X,X')\sim \DP\pa{\alpha \nu_0}\otimes \DP\pa{\beta \nu_0}}{\EEs{X}{f}\geq u, \EEs{X'}{f} \geq u}
\\&=\textstyle\nonumber
\PPs{(G,G')\sim \cG\pa{\alpha \nu_0}\otimes \cG\pa{\beta \nu_0}}{\int\pa{f - u}dG\geq 0, \int\pa{f - u}dG'\geq 0}
\\&\leq\textstyle\nonumber
\PPs{(G,G')\sim \cG\pa{\alpha \nu_0}\otimes \cG\pa{\beta \nu_0}}{\int\pa{f - u}(dG+dG')\geq 0}
\\&=\textstyle\nonumber
\PPs{G\sim \cG\pa{(\alpha+\beta) \nu_0}}{\int\pa{f - u}dG\geq 0}
\\&=
\exp\pa{h(\alpha+\beta)}.\label{mulpatern}
\end{align}
Then, by Fekete's superadditivity lemma, the limit $\lim_{\alpha\to\infty}h(\alpha)/\alpha$ exists and is equal to $\sup_{\alpha>0}h(\alpha)/\alpha$. Finally, from the last inequality in Theorem~\ref{thm:gen_ldp_DP}, we recover \eqref{rel:chernoff_(1)_lemma1} as
\[h(\alpha)/\alpha\leq \lim_{\alpha\to\infty}h(\alpha)/\alpha \leq -\Kinf(\nu_0, u, f).\]
In Section~\ref{sec:log_deviation}, we refine this approach by using a more sophisticated sequence of measures than \( (\alpha \nu_0)_\alpha \) in Theorem~\ref{thm:gen_ldp_DP}. This allows us to improve the bound by incorporating a unit mass perturbation \( \eta \) within the KL divergence.
In the subsequent section~\ref{sec:ex}, we will generalize this result to perturbations that go beyond the unit mass case, this time without relying on superadditivity but instead reducing back to the Beta distribution case.
\section{Perturbed KL bounds for DPs through superadditivity of the log deviation}
\label{sec:log_deviation}
We want to extend the perturbed KL bound of \cite{henzi2023some} from the Beta distribution to DPs. In an earlier version of their paper, \cite{henzi2023some} used a unit mass perturbation of $\eta=1$, which they were able to refine to $\eta=1+(a-1)/(b+1)$ in a subsequent (and current) version. In this section, we focus on extending this unit mass perturbed KL bound to the DP. 
\subsection{Half-spaces}
Our approach is straightforward: we first establish the superadditivity of the perturbed log deviation \( t \mapsto \log \PPs{X\sim\DP\pa{t \nu + \eta}}{\EEs{X}{f}\geq u} \) for \( t > 0 \), $\eta\in \cM(\Omega)$ with $\eta(\Omega)\leq 1$, and \( \nu \in \cM_1(\Omega) \). Then, after applying Fekete’s lemma, we reparametrize \( t \) and \( \nu \) in terms of \( \alpha \), \( \nu_0 \), and \( \eta \) to transfer the perturbation \( \eta \) inside the KL bound. The challenge now lies in handling the additional term \( \eta \) when proving superadditivity. Specifically, instead of a purely linear dependency on the scale parameter \( t \) in the Gamma process, we now have an affine function \( t \mapsto t\nu + \eta \), introducing an additional constant term. Our strategy, outlined in Lemma~\ref{lem:mumu}, is to establish a multiplicative behavior of the deviation probability --- when marginalized over the Gamma sample associated to the perturbation $\eta$ --- with respect to the Gamma samples associated to the linear component $t\nu$ (represented by the distributions \( \mu \) and \( \mu' \) in the lemma).

\begin{lemma}\label{lem:mumu}
Let $\eta\in \cM\pa{\Omega}$ such that $\eta\pa{\Omega}\leq 1$ and let $f,g\in \cC(\Omega)$. We define
$$\textstyle\ell(\mu)\triangleq\PPs{D\sim \cG\pa{\eta}}{\int fdD \geq \int g d\mu},$$ for $\mu\in \cM(\Omega)$. Then, for $\mu,\mu'\in \cM(\Omega)$,
$\ell(\mu+\mu')\geq \ell(\mu)\ell(\mu')$.
\end{lemma}
\begin{proof}
We first treat two simple cases: If $\int g {d\mu}\leq 0$, then $\ell(\mu+\mu')\geq \ell(\mu')\geq \ell(\mu)\ell(\mu').$ Similarly, if $\int g {d\mu'}\leq 0$, $\ell(\mu+\mu')\geq \ell(\mu)\geq \ell(\mu)\ell(\mu').$

We now consider the scenario where both $\int g {d\mu'}$ and $\int g {d\mu}$ are positive. We begin by considering a special parametrization of $D$:
Let $Z,B,E\sim \DP\pa{\eta}\otimes\text{Beta}\pa{\eta(\Omega),1-\eta(\Omega)}\otimes\text{Exp}(1)$, by considering the degenerate Beta distribution if $\eta\pa{\Omega}\in\sset{0,1}$. Since, $\text{Exp}(1) = \text{Gamma}(1, 1)$, we have $BE\sim \text{Gamma}(\eta(\Omega), 1)$. Thus $D=ZBE\sim \cG\pa{\eta}.$ We can therefore write
\begin{align}\textstyle\ell(\mu+\mu') = \PPs{Z,B,E}{BE\int f dZ
\geq {\int {g} d\pa{ \mu+\mu' }}
}=\EEs{Z,B}{h(BZ)^{\int {g} d\pa{ \mu+\mu' }}},\label{lem:supperadditivePbis,,eq:1}
\end{align}
where 
\[h(\nu)\triangleq
\left\{
    \begin{array}{ll}
        0 & \mbox{if }  \int{{f}d\nu }\leq 0\\
\exp\pa{-\pa{\int{{f}d\nu }}^{-1} }& \mbox{otherwise.}
    \end{array}
\right.
\]
From Proposition~\ref{prop:Ee^Z}, we have 
\[\eqref{lem:supperadditivePbis,,eq:1}\geq \EEs{Z,B}{{h(BZ)}^{\int {g} d{ \mu }}}\EEs{Z,B}{{h(BZ)}^{\int {g} d{ \mu' }}}{}=\ell\pa{\mu}\ell\pa{\mu'}.\]
\end{proof}

We now explain why the perturbation is limited by a mass of 1. In the above proof, we use the Gamma process representation of the DP. Thus, the perturbation has its own associated Gamma process, with total mass following a \( \text{Gamma}(\eta(\Omega),1) \) distribution (the random variable \( BE \)). For \( \eta(\Omega) = 1 \), this reduces to an exponential distribution. The use of the Beta random variable \( B \) is somewhat artificial but enables handling the more general case where \( \eta \) has a total mass less than 1. The result is obtained by leveraging the multiplicative property of the exponential distribution survival function.

We are now ready to state the main superadditivity result in Lemma~\ref{lem:supperadditiveP__}. The challenging part --- handling the perturbation \( \eta \) --- has already been addressed, leaving us with only the linear component \( t\nu \). Since Gamma processes have independent increments, the product of the \( s \) and \( t \) components interacts well with expectation, leading to a clean supperaditivity formulation of the log deviation.

\begin{lemma}[Supperaditivity of the log deviation with a unit mass perturbation] Let $f\in \cC(\Omega)$, $u\in \R$, $\nu\in \cM_1\pa{\Omega}$ and $\eta\in \cM\pa{\Omega}$ such that $\eta\pa{\Omega}\leq 1$. Then, the function \[t \mapsto \log \PPs{X\sim\DP\pa{t \nu + \eta}}{\EEs{X}{f}\geq u}\]
is superadditive on $(0,\infty)$.
\label{lem:supperadditiveP__}
\end{lemma}
\begin{proof} 
Consider $s,t>0$. We have 
\begin{align*}
        \PPs{X\sim \DP\pa{(s+t) \nu + \eta}}{\EEs{X}{f}\geq u}&=\textstyle\PPs{G,G',D\sim \cG\pa{s \nu}\otimes\cG\pa{t \nu}\otimes\cG\pa{\eta}}{\int (f-u)dD \geq \int (u-f)d(G+G')}
        \\&=\textstyle  \EEs{G,G'\sim \cG\pa{s \nu}\otimes\cG\pa{t \nu}}{\ell\pa{G+G'}},
\end{align*}
where $\ell$ is defined as in Lemma~\ref{lem:mumu}, with the two functions $f-u$ and $u-f$ respectively.
Applying the result of Lemma~\ref{lem:mumu}, we have \begin{align*}\EEs{G,G'\sim \cG\pa{s \nu}\otimes\cG\pa{t \nu}}{\ell\pa{G+G'}}&\geq \EEs{G,G'\sim \cG\pa{s \nu}\otimes\cG\pa{t \nu}}{\ell\pa{G}\ell\pa{G}}\\&=
\PPs{X\sim \DP\pa{s \nu + \eta}}{\EEs{X}{f}\geq u}
\PPs{X\sim \DP\pa{t \nu + \eta}}{\EEs{X}{f}\geq u},
\end{align*}
Where the last equality is because $G \indep G'$. 
\end{proof}

We can now formulate Theorem~\ref{th:devlinoffset}, which provides a first generalization of \eqref{rel:chernoff_(1)_lemma1}, relying on Lemma~\ref{lem:supperadditiveP__} and Fekete’s superadditivity lemma. 
We can interpret this bound as permitting a minor perturbation (played by the distribution $\eta$) around the base distribution $\nu_0$ within $\Kinf$.
This result is indeed sharper than \eqref{rel:chernoff_(1)_lemma1} as $\eta$ being the zero measure is allowed in the supremum. Notice that the optimal \(\eta^*\) is straightforward to determine: given the constraint, it must concentrate its mass on the larger positive values of \( f - u \). In particular, when applied to the Beta distribution, this recovers the result from the earlier version of \cite{henzi2023some}, where a unit mass perturbation (\(\eta = 1\)) was used, under the assumption that \( a \geq 1 \). However, we do not require this assumption --- when \( a < 1 \), we simply adjust the perturbation to be equal to \( a \).

\begin{theorem}
[Unit mass perturbed KL deviation bound] Let $f\in \cC\pa{\Omega}$ and $u\in [f_{\min},f_{\max})$, where $f_{\max}\triangleq \max_{x\in \Omega} f(x)$, $f_{\min}\triangleq \min_{x\in \Omega} f(x)$. Then, 
\[\log \PPs{X\sim\DP(\alpha \nu_0)}{\EEs{X}{f}\geq u}
    \leq - \max_{\substack{{\eta\in \cM(\Omega),} \\ {\eta(\Omega)\leq 1, ~ \eta\leq \alpha \nu_0}}} \pa{\alpha - \eta(\Omega)}\Kinf\pa{\frac{\alpha \nu_0 - \eta}{\alpha - \eta(\Omega)}, u, f}.\]
In addition, a maximizer $\eta^*\in \cM(\Omega)$ is computed as 
\[\eta^* = \left\{
    \begin{array}{ll}
    \alpha\nu_0\pa{~\cdot~\cap A_1} & \mbox{if } \alpha\nu_0(A_1) = 1, 
    \\
    \alpha\nu_0\pa{~\cdot~\cap A_1} + \frac{1-\alpha\nu_0\pa{A_1}}{\nu_0\pa{A_0}}\nu_0\pa{~\cdot~\cap A_0} & \mbox{if } \alpha\nu_0(A_1) < 1,
    \end{array}
\right.\]
where $u_1\triangleq \inf\sset{x \geq u,~\alpha\nu_0\pa{f^{-1}\pa{(x,\infty)}} \leq 1},$  
$A_0=f^{-1}(\sset{u_1})$ and 
$A_1\triangleq f^{-1}\pa{(u_1,\infty)}$.
\label{th:devlinoffset}
\end{theorem}
\begin{proof}
Using Lemma~\ref{lem:supperadditiveP__}, and Fekete's lemma, we have
\[\lim_{t\to\infty}\log \PPs{X\sim\DP\pa{t \nu + \eta}}{\EEs{X}{f}\geq u}/t = \sup_{t>0}\log \PPs{X\sim\DP\pa{t \nu + \eta}}{\EEs{X}{f}\geq u}/t.\]
On the other hand, from Theorem~\ref{thm:gen_ldp_DP}, we have
\[\lim_{t\to\infty}\log \PPs{X\sim\DP\pa{t \nu + \eta}}{\EEs{X}{f}\geq u}/t \leq -\Kinf\pa{\nu,u,f},\]
so, for all $t>0$, 
\[\log\PPs{X\sim\DP\pa{t \nu + \eta}}{\EEs{X}{f}\geq u} \leq {-t\Kinf\pa{\nu,u,f}}.\]
This is true for all $\eta\in \cM\pa{\Omega}$ such that $\eta\pa{\Omega}\leq 1$, so in particular,
we can consider some $\eta\leq \alpha \nu_0$. 
Then, letting $t$ be $\alpha - \eta\pa{\Omega}$ and $\nu$ be $(\alpha\nu_0 - \eta)/t$, we have
\[\log\PPs{X\sim\DP\pa{\alpha \nu_0}}{\EEs{X}{f}\geq u} \leq {-(\alpha - \eta\pa{\Omega})\Kinf\pa{\frac{\alpha\nu_0 - \eta}{\alpha - \eta\pa{\Omega}},u,f}}.\]
 Taking the $\inf$ over $\eta$, we get the result. 
 
 To get the maximizer $\eta^*$, we use Lemma~\ref{lem:Variational formula for Kinf}:
 \begin{align*}
     &\sup_{\substack{{\eta\in \cM(\Omega),} \\ {\eta(\Omega)\leq 1, ~ \eta\leq \alpha \nu_0}}}(\alpha - \eta\pa{\Omega})\Kinf\pa{\frac{\alpha\nu_0 - \eta}{\alpha - \eta\pa{\Omega}},u,f} \\&= \max_{\lambda\in [0,1/(f_{\max} - u)]} \sup_{\substack{{\eta\in \cM(\Omega),} \\ {\eta(\Omega)\leq 1, ~ \eta\leq \alpha \nu_0}}} \int{\log\pa{1-\lambda(f-u)}}d\pa{\alpha\nu_0 - \eta}
     \\&= \max_{\lambda\in [0,1/(f_{\max} - u)]}\int{\log\pa{1-\lambda(f-u)}}{\alpha d\nu_0} - \underbrace{\inf_{\substack{{\eta\in \cM(\Omega),} \\ {\eta(\Omega)\leq 1, ~ \eta\leq \alpha \nu_0}}} \int \log\pa{1-\lambda(f-u)}d \eta}_{\int \log\pa{1-\lambda(f-u)}d \eta^*},
 \end{align*}
 where the last equality follows from the observation that \( -\log(1 - \lambda (f - u)) \geq 0 \) if and only if \( f \geq u \). This implies that the minimizer \( \eta \) must prioritize placing its mass on larger values of \( f \) (while respecting the unit mass constraint and remaining below \( \alpha \nu_0 \)) and must avoid assigning any mass to values of \( f \) lower than \( u \).
\end{proof}

\subsection{Closed convex sets}

If \( \eta \) has mass 1 and is concentrated at a single point \( x \in \Omega \), we can extend Lemma~\ref{lem:mumu} into Lemma~\ref{lem:Chyp}, 
(and Lemma~\ref{lem:supperadditiveP__} into Lemma~\ref{lem:supperadditiveP_}),
which generalizes the previous multiplicative behavior to deviations toward a closed convex set $C\subset \cM_1(\Omega)$.   
However, the proof breaks down when \( \eta \) is not concentrated at a single point or has a total mass less than 1, because in this case, the function \( \ell \) can no longer be expressed as an expectation of a power function (as in \eqref{lem:supperadditivePbis,,eq:1}), which was previously easier to handle. The approach we take here remains similar but relies on both sides of the exponential distribution, making the multiplicative pattern more challenging to establish.  

Lemma~\ref{lem:Chyp} could also be used to derive a superadditivity property and, consequently, a bound for the deviation toward a closed convex set. However, such a result is weaker than what we obtain with half-spaces. Indeed, the minimizer over $C$ within the KL lies on a supporting hyperplane of both \( C \) and a sublevel set of the KL.\footnote{Since the KL is convex, it cannot have a local minimum
inside $C$ if its global minimum lies outside $C$. Thus, the minimizer is the intersection of 2 convex sets: \( C \) and some closed sublevel set of the KL. We thus get the supporting hyperplane using the geometric Hahn–Banach theorem.} This allows us to bound the probability of belonging to \( C \) by the probability of belonging to the corresponding closed half-space that contains~\(C\). In turn, from Theorem~\ref{th:devlinoffset}, this probability is bounded using the associated \( \Kinf \), which coincides with the infimum of the KL over \( C \). 
Nevertheless, the superadditivity result of Lemma~\ref{lem:supperadditiveP_} remains valuable in its own right. For example, it implies, via Fekete's lemma, the existence of an asymptotic equivalent for the log deviation as the scale tends to \( \infty \).
Moreover, in Remark~\ref{rk:Dirichlet-Multinomial correspondence}, we will explore a special case where superadditivity can be established more directly.

\begin{lemma}\label{lem:Chyp}Let $C\subset \cM_1(\Omega)$ be a closed convex set and
let $x\in \Omega$. We define
$$\ell(\mu)\triangleq\PPs{E\sim \text{Exp}\pa{1}}{\frac{\mu+E\delta_x}{\mu(\Omega)+E}\in C},$$ for $\mu\in \cM(\Omega)$, where $\delta_x$ is the Dirac unit mass concentrated at $x$. Then, for $\mu,\mu'\in \cM(\Omega)$,
$\ell(\mu+\mu')\geq \ell(\mu)\ell(\mu')$.
\end{lemma}
\begin{proof}We represent $C$ as the intersection of its supporting half-spaces \cite{boyd}:
     There is a set $F_C\subset \cC(\Omega)$ such that $p\in C \iff \forall f\in F_C$, $\EEs{p}{f}\geq 0$.
     Therefore, $(\mu+E\delta_x)/(\mu(\Omega)+E)\in C $ if and only if for all $ f\in F_C,$ one of the following is true
    \begin{align*}
       E\geq  -\frac{\int f d\mu}{f(x)}\quad\mbox{if }f(x)>0,\quad\quad&
        \textstyle{\int f d\mu}\geq 0\quad\mbox{if }f(x)=0,\quad\quad&
        E\leq  -\frac{\int f d\mu}{f(x)}\quad\mbox{if }f(x)<0.
    \end{align*}
So, with the convention that the infimum (resp. supremum) over the empty set is $\infty$ (resp. $-\infty$), we define 
\begin{align*}
    A(\mu)&\triangleq \textstyle\II{{\int f d\mu\geq 0~\forall f \in F_C \mbox{ s.t. }f(x)=0}},\\
    M_1(\mu)&\triangleq \pa{\sup_{f\in F_C,~f(x)>0}-\frac{\int f d\mu}{f(x)}}^+,\quad\quad
    M_2(\mu)\triangleq \inf_{f\in F_C,~f(x)<0}-\frac{\int f d\mu}{f(x)},
\end{align*}
and get
$$\ell(\mu)={A(\mu)}\PPs{E\sim \text{Exp}\pa{1}}{M_1(\mu)\leq E \leq M_2(\mu)} = A(\mu)\pa{
e^{-M_{1}\pa{ \mu }}
-
e^{-M_{2}\pa{ \mu }}
}^+.$$
Clearly, $A(\mu+\mu')\geq A(\mu)A(\mu')$. On the other hand, as $M_1$ is subadditive and $M_2$ is superadditive, 
we have
\[\pa{
e^{-M_{1}\pa{ \mu + \mu' }}
-
e^{-M_{2}\pa{ \mu + \mu' }}
}^+ \geq \pa{
e^{-M_{1}\pa{\mu }}
e^{-M_{1}\pa{\mu' }}
-
e^{-M_{2}\pa{\mu }}
e^{-M_{2}\pa{\mu' }}
}^+,\]
which is itself, thanks to Proposition~\ref{prop:quicka_0,a_1,b_0,b_1}, lower bounded by
\[ 
\pa{
e^{-M_{1}\pa{\mu }}
-
e^{-M_{2}\pa{\mu }}
}^+
\pa{
e^{-M_{1}\pa{\mu' }}
-
e^{-M_{2}\pa{\mu' }}
}^+.
\]
\end{proof}


\begin{lemma}[Supperaditivity for convex sets with a unit mass perturbation] Let $C$ be a closed convex set of $\cM_1\pa{\Omega}$. Let $\nu\in \cM_1\pa{\Omega}$ and $x\in\Omega$.  Then, the functions \[t\mapsto \log \PPs{\DP\pa{t\nu+\delta_x}}{C}, \quad t\mapsto \log \PPs{\DP\pa{t\nu}}{C} \]
are superadditive on $(0,\infty)$, where $\delta_x$ is the Dirac unit mass concentrated at $x$.
\label{lem:supperadditiveP_}
\end{lemma}

\begin{proof}
Consider $s,t>0$. We use the Gamma process representation, considering ${G,G',E\sim \cG\pa{s \nu}\otimes\cG\pa{t \nu}\otimes\text{Exp}(1)}.$ 

The proof for the first function 
    relies on Lemma~\ref{lem:Chyp}. Indeed, \begin{align*}\PPs{\DP\pa{(s+t)\nu+\delta_x}}{C} &= \PPs{G,G',E\sim \cG\pa{s \nu}\otimes\cG\pa{t \nu}\otimes\text{Exp}(1)}{\frac{G+G'+E\delta_x}{G(\Omega)+G'(\Omega)+E}\in C}\\
    &=\EEs{G,G' \sim \cG\pa{s \nu}\otimes\cG\pa{t \nu}}{\ell\pa{G+G'}} \\&\geq \EEs{G,G' \sim \cG\pa{s \nu}\otimes\cG\pa{t \nu}}{\ell\pa{G}\ell\pa{G'}}=\PPs{\DP\pa{s\nu+\delta_x}}{C}\PPs{\DP\pa{t\nu+\delta_x}}{C},\end{align*}
where $\ell$ is defined as in Lemma~\ref{lem:Chyp}.
    
For the second function, we can directly obtain the result effortlessly, leveraging the convexity of \( C \). Specifically, if \( G/G(\Omega) \) and \( G'/G'(\Omega) \) belong to \( C \), it follows that \( (G+G')/(G(\Omega)+G'(\Omega)) \in C \).
\end{proof}

\begin{remark}[Link with the Dirichlet-Multinomial correspondence]\label{rk:Dirichlet-Multinomial correspondence}
A rather simple scenario is the case of a Dirichlet distribution (i.e., when $\Omega=[d]$ for some $d\in \N^*$). In this context, we can consider the result of Lemma~\ref{lem:supperadditiveP_}, with the convex set $C=C_p$, for some fixed $p\in \cM_1\pa{[d]}$, defined as \[\textstyle C_p\triangleq\sset{q\in\cM_1\pa{[d]},~ \forall i \in [d-1],~\sum_{j>i}\pa{q_j - p_j} \geq 0}.\]
This specific instance is interesting when $\alpha \nu_0\pa{\sset{i}} \in \N$ for all $i\in [d]$, and $\alpha \nu_0\pa{\sset{d}} \geq 1$, because an alternative method for proving superadditivity emerges, involving the use of the multinomial distribution. 
More precisely, based on the following Fact~\ref{fact:Dirichlet-Multinomial correspondence}, for $x=(\alpha \nu_0 - \delta_d)/\pa{\alpha - 1}$ and for $\delta_d\in \cM_1([d])$ being the Dirac unit mass concentrated at the last index $d\in [d]$, the function
\(n\to\log \PPs{\DP(n x + \delta_d)}{C_p}\)
is superadditive for all $n\in \N$ such that $nx\in \N^d$. Indeed, using the representation of the multinomial random variable as a sum of categorical random variables and letting $(X_{\ell})_{\ell\in \N}\overset{iid}{\sim}\text{Cat}(p)$, we have
\begin{align*}
    \PPs{\DP((n+m) x + \delta_d)}{C_p} &= \textstyle\PP{\forall i\in [d-1],~ \sum_{j\in [i]} \sum_{\ell \in [n+m]} \pa{X_{j,\ell} - x_j} \geq 0 }\\&\geq  \PPs{\DP(n x + \delta_d)}{C_p} \PPs{\DP(m x + \delta_d)}{C_p}.
\end{align*}
\noindent
We thus have, from the superadditive lemma and Theorem~\ref{thm:gen_ldp_DP}, that $$\frac{1}{n}\log \PPs{\DP(n x + \delta_d)}{C_p}\leq \lim_{n\to \infty} \frac{1}{n}\log \PPs{\DP(n x + \delta_d)}{C_p} \leq - \inf_{\mu\in C_p}\KL{x}{\mu},$$ for all $n\in \N$ such that $nx\in \N^d$.
Finally, by setting \( n = \alpha - 1 \) and multiplying both sides by \( n \), we recover a bound where the perturbation is incorporated into the KL term.
\end{remark}

\begin{fact}[Dirichlet-Multinomial correspondence \cite{Chafa__2009}]
\label{fact:Dirichlet-Multinomial correspondence}
Let $p\in \cM_1\pa{[d]}$ and $n\in \N$. We consider a sequence of integers $k_0=0\leq k_1\leq \dots\leq k_{d-1}\leq n \leq k_{d} = n+1$. Let $M\sim\text{Multinomial}(n,p)$ and $D\sim \text{Dir}((k_{i} - k_{i-1})_{i\in [d]})$. Then, 
\[\textstyle\PP{\forall i \in [d-1],~\sum_{j\in [i]} M_j \geq k_i} = \PP{\forall i \in [d-1],~\sum_{j>i} D_j \geq \sum_{j>i} p_j}.\]
\end{fact}



\section{Beyond unit mass perturbation}
\label{sec:ex}
In this section, we aim to investigate perturbations exceeding the unit mass for the Beta distribution, the Dirichlet distribution, and the DP. 
\subsection{Beta distribution}
We consider the Beta distribution \(\text{Beta}(a, b),\) with parameters \( a, b > 0 \).
We begin by stating a more general Beta bound than that of \cite{henzi2023some} in Lemma~\ref{lem:citephenzi2023some}. Indeed, their approach does not explicitly isolate a perturbation \( \eta \); instead, they prove the result by directly substituting the value \( \eta = 1 + \frac{a - 1}{b + 1} \). In contrast, we allow \( \eta \) to vary within \( [0, \min(a, 1 + \frac{a - 1}{b + 1})] \), 
which enables us to remove the constraint that \( a \) must be greater than 1. We go even further: through a more refined analysis, we establish that the upper bound of the interval for $\eta$ can depend not only on \( a \) and \( b \) but also on \( u \), enabling a significantly larger perturbation.
\begin{lemma}[Beyond unit mass perturbation for the Beta distribution]
Let $a>0$, $b > 0$, $u\in [0,1]$, $S_i$ for $i\in \sset{0,1,2,\infty}$ as defined in Proposition~\ref{prop:nonincreasing} and $\eta\in [0,\min(a,S_{\infty}(a,b,u))]$. We consider a sample $B\sim\text{Beta}(a,b)$. Then, if $u> (a-\eta)/(a+b-\eta)$,
\[\log\PPs{}{B\geq u} - \log\PP{B\geq (a-\eta)/(a+b-\eta)} \leq -(a+b-\eta)\kl{(a-\eta)/(a+b-\eta)}{u}.\]
In particular, we can take
$$
\eta = \left\{
    \begin{array}{ll}
        S_{i}(a,b,u) & \mbox{if } a\geq 1 \\
        a & \mbox{otherwise,}
    \end{array}
\right.
$$
for all $i\in \sset{0,1,2,\infty}$, depending on whether we prioritize optimality (\( i \in \{2, \infty\} \)) or the simplicity of formulation and computation (\( i \in \{0, 1\} \)).
\label{lem:citephenzi2023some}
\end{lemma}
\begin{proof} 
Let $B\sim\text{Beta}\pa{tx + \eta, t(1-x)}$, with $t=a+b-\eta$ and $x=\frac{a-\eta}{a+b-\eta}$. We assume that $u>x$.
Consider $g(v)\triangleq \PP{B \geq x}\exp\pa{-t \kl{x}{v}}$ and $h(v)\triangleq \PP{B \geq v}$. We have $g(x)=h(x)$ and $g(1)=h(1)=0$. Moreover, 
\begin{align*}h'(v) &= \frac{-\Gamma\pa{t+\eta}}{\Gamma\pa{tx + \eta}\Gamma\pa{t(1-x)}}v^{tx + \eta -1}(1-v)^{t(1-x) - 1}\leq 0,\\g'(v) &= \PP{B \geq x} t\pa{\frac{v}{x}}^{-\eta}\pa{ \frac{1-v}{1-x}- \frac{v}{x} }\pa{\frac{v}{x}}^{tx + \eta - 1} \pa{\frac{1-v}{1-x}}^{t(1-x) - 1} \\&= \underbrace{\PP{B \geq x} \frac{tx^{\eta - 1}\Gamma\pa{tx + \eta}\Gamma\pa{t(1-x)}}{(1-x)\Gamma\pa{t+\eta}} {v}^{-\eta}\pa{v - {x}{}}}_{J(v)} h'(v),\\J'(v) &= \PP{B \geq x}  \frac{tx^{\eta - 1}\Gamma\pa{tx + \eta}\Gamma\pa{t(1-x)}}{(1-x)\Gamma\pa{t+\eta}}\pa{1+\eta\pa{\frac{x}{v} - 1} }v^{-\eta}.
\end{align*}
We claim that the proof is complete if we establish that \( 1 - J \) either maintains a constant sign on \( (x, u) \) or changes sign exactly once on \( (x, u) \) but not on \( (u,1) \). In the first case, we obtain  
\[\textstyle
g(u)-h(u) = \int_{x}^u \pa{1-J(v)}(-h'(v))dv \geq 0,
\]  
while in the second case, we have  
\[\textstyle
g(u)-h(u) = \int_{u}^1 \pa{J(v)-1}\pa{-h'(v)}dv \geq 0.
\]
If $\eta \in [0, 1/(1-x)]$, then we see that $J'\geq 0$, so \( J - 1 \) changes its sign at most once on \( (x,1) \). If $\eta>1/(1-x)$, \( J \) increases from \( 0 \) to a maximum over the segment $[x,\frac{x}{1-1/\eta}]$ and then decreases on $[\frac{x}{1-1/\eta}, 1]$. If \( J(1) > 1 \), then again \( J - 1 \) changes its sign at most once on \( (x,1) \). Now, if $J(1)\leq 1$, we have $J(u)=J(1)\frac{u-x}{1-x}u^{-\eta}\leq \frac{u-x}{1-x}u^{-\eta},$ which is upper bounded by~$1$ from Proposition~\ref{prop:nonincreasing}. Plus, isolating $x$, we get $\frac{u-u^{\eta}}{1-u^{\eta}} \leq x$. Thus, from the inequality $u^{1-\eta}+(\eta-1)u-\eta \geq 0$ (obtained by noting that the derivative with respect to \( u \) is negative), we get 
\(
 (1-1/\eta)u \leq \frac{u-u^{\eta}}{1-u^{\eta}} \leq x ,
\)
so $u\in [x,\frac{x}{1-1/\eta}]$, i.e., $1-J\geq 0$  on $[x,u]$.
\end{proof}

\begin{corollary}
\label{cor:citephenzi2023some}
Let $u,v,w\in  [0,1],$ with $v>u$. Let $S_i$ for $i\in \sset{0,1,2,\infty}$ as defined in Proposition~\ref{prop:nonincreasing}.
Let $\eta\in [0,S_\infty(a,b,u/v)].$ From point $(v)$ of Proposition~\ref{prop:nonincreasing}, since $\frac{(u-w)^+}{v-u+(u-w)^+}\leq \frac{u}{v}$, we have  $S_\infty\pa{a,b,\frac{(u-w)^+}{v-u+(u-w)^+}}\geq S_\infty\pa{a,b,\frac{u}{v}}$. Thus, from Lemma~\ref{lem:citephenzi2023some},
\begin{align*}\log\PP{Bv + (1-B)w \geq u} &= \log\PP{B \geq {\frac{(u-w)^+}{v-u+(u-w)^+}}} \\&\leq -(a+b-\eta)\kl{\frac{a-\eta}{a+b-\eta}}{{ \frac{(u-w)^+}{v-u+(u-w)^+}}} \\&=-(a+b-\eta)\Kinf\pa{\begin{pmatrix} \frac{a-\eta}{a+b-\eta} \\ \frac{b}{a+b-\eta} \end{pmatrix} , u, \begin{pmatrix}   v \\ w \end{pmatrix}}.\end{align*}
\end{corollary}

\subsection{Dirichlet process}
Lemma~\ref{lem:citephenzi2023some} and Corollary~\ref{cor:citephenzi2023some} serves as our starting point for extending the result to DPs. The intuition for moving on to the generalization (stated in Theorem~\ref{thm:beyound1offset}) is to consider the measurable partition $\Omega = I\sqcup J$, where $I$ is the support of the perturbation. Then, consider the beta sample defined by 
$B=X(I)$, where $X\sim\DP\pa{\alpha \nu_0}$, 
and see that $\EEs{X}{f}=B \EEccs{X}{f}{I} + (1-B)\EEccs{X}{f}{J}$. The trick is then to see that $B$ and $\EEccs{X}{f}{I}$ (resp. $\EEccs{X}{f}{J}$) are independent (so that the previous lemma can be used). Indeed, if we use $G\sim \cG\pa{\alpha \nu_0}$ to generate $X=G/G(\Omega)$, then it's known that $G\pa{\cdot \cap I}/G\pa{I}$ (resp. $G\pa{\cdot \cap J}/G\pa{J}$) is independent of $(G\pa{I},G(J))$ \cite{lukacs1955characterization}. Then, what we obtain is an expectation (over some non-perturbed DP) of some exponential bound, with a binary KL as a rate function. Leveraging the connection between \( \Kinf \) and the MGF of the Gamma process, we can convert this into the desired  \( \Kinf \) exponential bound.

In Theorem~\ref{thm:beyound1offset}, we focus on perturbations that are essentially singletons (even if they involve multiple points, the possible values of \( f \) remain the same). In contrast, Theorem~\ref{th:devlinoffset} allows the perturbation to be spread out. The reason for this restriction is similar to what we encountered in the convex case. Specifically, when the perturbation is spread out, a single value \( v \) is replaced by an expectation of \( f - u \) throughout the proof. In the convex case, this prevented us from leveraging the multiplicative behavior, as it was difficult to pass it through this expectation. Here, this prevents us from fully eliminating the perturbation, as the distribution in the expectation of \( f - u \) depends on the perturbation.
Our approach to overcome this is to exploit the fact that when there is only one possible value, all expectations coincide. This allows us to replace the expectation involving the perturbation with that of a non-perturbed DP.

An alternative approach would be to first upper bound the log deviation by considering a perturbation concentrated at a single higher value. However, this would also affect how we apply the previous lemma. Indeed, the total allowed mass is constrained by a function of the mass of the non-perturbed DP parameter at the concentration point --- the larger this mass, the better the bound. This means we cannot arbitrarily shift the perturbation’s mass without altering how the lemma applies, which would disrupt the result we aim to establish.

\begin{theorem}
[DP deviation bound: beyond unit mass perturbation] 
Consider $f\in\cC\pa{\Omega}$ and $u\in [f_{\min},f_{\max})$, where $f_{\max}\triangleq \max_{x\in \Omega} f(x)$, $f_{\min}\triangleq \min_{x\in \Omega} f(x)$. Let 
$$
M_i(a,b,\cdot) = \left\{
    \begin{array}{ll}
        S_{i}(a,b,\cdot) & \mbox{if } a\geq 1 \\
        a & \mbox{otherwise,}
    \end{array}
\right.
$$
for all $i\in \sset{0,1,2,\infty}$, where $S_i$ for is as defined in Proposition~\ref{prop:nonincreasing}.
Then, for all $i\in \sset{0,1,2,\infty}$,
\[\log \PPs{X\sim\DP(\alpha \nu_0)}{\EEs{X}{f}\geq u}\leq -\sup_{v > u}{\pa{{\alpha - \eta_v\pa{\Omega}}}\Kinf\pa{\frac{\alpha\nu_0-\eta_{v}}{\alpha - \eta_v\pa{\Omega}}, u, f}},\]
where $\eta_v \triangleq
M_i(\alpha \nu_0\pa{f^{-1}(\sset{v})},\alpha - \alpha \nu_0\pa{f^{-1}(\sset{v})},u/v)\frac{\nu_0\pa{~\cdot~\cap f^{-1}(\sset{v})}}{\nu_0\pa{f^{-1}(\sset{v})} }.
$ 
\label{thm:beyound1offset}
\end{theorem}

\begin{proof}
Consider some $v > u$. We define $t\triangleq \alpha - \eta_v\pa{\Omega}$ and $\nu\triangleq (\alpha \nu_0 - \eta_v)/(\alpha - \eta_v\pa{\Omega})$. Let $I\triangleq f^{-1}\pa{\sset{v}}$ and $J\triangleq \Omega \backslash I$. We have
\begin{align}
    \nonumber&\PPs{X\sim\DP(t\nu + \eta_v)}{\EEs{X}{f}\geq u} \\&=\textstyle
    \nonumber\PPs{G\sim \cG\pa{t\nu + \eta_v}}{{\int fdG}{} \geq u}
    \\&=
    \nonumber\PPs{G\sim \cG\pa{t\nu +\eta_v}}{\frac{G(I)}{G\pa{\Omega}} v +\frac{G(J)}{G(\Omega)}\frac{\int_J fdG}{G(J)} \geq u}
     \\&=
    \nonumber\PPs{X,B\sim \DP\pa{t\nu}\otimes \text{Beta}\pa{(t\nu +\eta_v)(I),t\nu(J)}}{B v +(1-B)\EEccs{X}{f}{J} \geq u}
    \\&
    \label{citephenzi2023some}
    \leq \EEs{X\sim \DP\pa{t\nu}}{\exp\pa{-t\Kinf\pa{\begin{pmatrix}
        \nu(I) \\ \nu(J)
    \end{pmatrix}
  , u, \begin{pmatrix}
        v \\ \EEccs{X}{f}{J}
    \end{pmatrix}}{}}},
    \\&\leq\label{sup_PI}
    \exp\pa{-t\Kinf\pa{\nu,u,f}}.
    \end{align}
    where \eqref{citephenzi2023some} is from Corollary~\ref{cor:citephenzi2023some} and \eqref{sup_PI}
is from Proposition~\ref{prop:sup_PI}.
\end{proof}

\section{Conclusion}
In this work, we refine the bounds for DPs by introducing a perturbation into the KL bound. Leveraging superadditivity and Fekete’s lemma, we derive a tighter non-asymptotic deviation bound by incorporating a unit mass perturbation into the KL divergence. Additionally, we explore perturbations beyond the unit mass case without relying on superadditivity.  Future directions include the following open questions:
\begin{itemize}
    \item In Theorem~\ref{thm:beyound1offset}, since the perturbation must be applied to a single value of \( f \), determining where to apply the perturbation becomes less clear (particularly when the point \( v \) maximizing \( f \) has insufficient mass). In contrast, in Theorem~\ref{th:devlinoffset}, the perturbation could be spread, making it straightforward to prioritize the highest values of \( f \) first.  Moreover, if \( \nu_0 \) assigns no atom to values of \( f \) greater than \( u \), the total perturbation cannot exceed 1 using our results.
    \item In the same vein, a similar issue arises for superadditivity in the convex case: the perturbation must be applied to a single point, with the additional constraint that the mass must be exactly 1. We believe this last limitation is an artifact of our analysis and should be avoidable. One reason for this is that we have superadditivity in both the case of zero mass and the case of mass 1, making it reasonable to expect superadditivity for intermediate cases as well.
\end{itemize}


\bibliographystyle{amsplain}
\bibliography{ref}


\begin{supplement}
\appendix
\section{Miscellaneous}
\begin{proposition}
\label{prop:nonincreasing}
 Let $a>0$, $b > 0$, $u\in (0,1)$ and 
\begin{align*} 
R(\eta)&\triangleq\frac{u-\frac{a-\eta}{a+b-\eta}}{1-\frac{a-\eta}{a+b-\eta}}u^{-\eta} -1 = \frac{ub-{(1-u)(a-\eta)}}{b}u^{-\eta} -1.
\end{align*}
Then,
\begin{itemize}
    \item[$(i)$] In $\R_+$, $R$ admits a unique root $\eta^* = S_{\infty}(a,b,u)\triangleq a - b\frac{u}{1-u} + \frac{W_0\pa{ \frac{b u^{a - b\frac{u}{1-u}} }{1 - u}\log\pa{\frac{1}{u}}}}{\log\pa{\frac{1}{u}}}, $ where $W_0$ is the principal branch of the Lambert W function.
    \item[$(ii)$] $R(\eta)\leq 0$ for $\eta\in [0,\eta^*]$.
    \item[$(iii)$] $\eta^*\leq a \iff a\geq 1$.
    \item[$(iv)$] If $a\geq 1$, we have $\eta^*\geq S_2(a,b,u)\geq S_1(a,b,u)\geq S_0(a,b,u)\geq 0,$ where
\begin{align*} 
S_2(a,b,u)&\triangleq a-b\pa{b\log\pa{\frac{1}{u}}+\frac{1}{u}-1}\frac{\sqrt{1+\frac{2\pa{\frac{1}{u}-1}^{2}\log\pa{\frac{1}{u}}(a-1)}{\pa{b\log\pa{\frac{1}{u}}+\frac{1}{u}-1}^2}}-1}{\pa{\frac{1}{u}-1}^{2}},
\\S_1(a,b,u)&\triangleq a-b\frac{a-1}{b+\pa{\frac{1}{u}-1}/\log\pa{\frac{1}{u}}},\\S_0(a,b,u)&\triangleq a-b\frac{a-1}{b+1}.
\end{align*}
    \item[$(v)$] If $a\geq 1$, $S_\infty(a,b,\cdot)$ is non-increasing.
\end{itemize}
\end{proposition}
\begin{proof}
Below, we provide the proofs for each point individually.
\begin{itemize}
\item[$(i)$] The equation $R(\eta)=0$ is equivalent to
\begin{align*}
    \pa{b\frac{u}{1-u}+\eta -a}\log\pa{\frac{1}{u}}\exp\pa{\pa{\eta-a+b\frac{u}{1-u}}\log\pa{\frac{1}{u}}} = \frac{b{u}^{a-b\frac{u}{1-u}}}{1-u}\log\pa{\frac{1}{u}}.
\end{align*}
Thus, we have that $\pa{b\frac{u}{1-u}+\eta -a}\log\pa{\frac{1}{u}}$ is the unique solution to a transcendental equation involving the Lambert W function. Since the RHS is non-negative, we have \[\pa{b\frac{u}{1-u}+\eta -a}\log\pa{\frac{1}{u}} = W_0\pa{ \frac{b u^{a - b\frac{u}{1-u}} }{1 - u}\log\pa{\frac{1}{u}}}.\]
leading to the expression of $\eta^*$. 
\item[$(ii)$] Since \( R(0) = \frac{b u - a(1 - u)}{b} - 1 < 0 \) and \( \eta^* \) is unique, it follows by continuity of $R$ that \( R(\eta) \leq 0 \) for \( \eta\in [0, \eta^*] \).
\item[$(iii)$] We have 
\begin{align*}
    \eta^* \leq a &\iff  {W_0\pa{ \frac{b u^{a - b\frac{u}{1-u}} }{1 - u}\log\pa{\frac{1}{u}}}}{} \leq \log\pa{\frac{1}{u}} b\frac{u}{1-u}\\
    &\iff  \frac{b u^{a - b\frac{u}{1-u}} }{1 - u}\log\pa{\frac{1}{u}} \leq \log\pa{\frac{1}{u}} b\frac{u}{1-u}\exp\pa{\log\pa{\frac{1}{u}} b\frac{u}{1-u}}
    \\&\iff u^a\leq u
    \iff a \geq 1.
\end{align*}

\item[$(iv)$] Assume $a\geq 1$. Taking the logarithm of both sides of the equation $R(\eta^*)=0$, we obtain:
\[\log(u){\eta^*} = \log(u) +\log\pa{1-\frac{\frac{1-u}{u}(a-\eta^*)}{b}}.\]
For $x\geq 0$, from the relation $\log(1-x)\leq-x -x^2/2\leq -x$, we have
\[ {\eta^*} \geq 1 +\frac{{{\frac{1-u}{ub}(a-\eta^*)}{}}}{\log(1/u)} + \frac{{\pa{\frac{1-u}{ub}(a-\eta^*)}^2{}}}{2\log(1/u)}\geq 1 +\frac{{{\frac{1-u}{ub}(a-\eta^*)}{}}}{\log(1/u)}.\]
In other words, \( \eta^* \) is at least as large as the smallest root of a quadratic polynomial, which in turn is greater than or equal to the root of a linear polynomial.
\begin{align*}\eta^*&\geq a-b\frac{\sqrt{2(1/u-1)^{2}\log(\frac{1}{u})(a-1)+\left( b\log\left(\frac{1}{u}\right)+1/u-1\right)^{2}}-\pa{b\log(\frac{1}{u})+1/u-1}}{(1/u-1)^{2}}
\\&\geq a-b\frac{a-1}{b+(1/u-1)/\log(1/u)}\geq a-b\frac{a-1}{b+1}.
\end{align*}

\item[$(v)$] Assume $a\geq 1$. For $\eta\in [0,a]$ and $u\in\pa{0,1}$, let \[f(\eta,u)\triangleq 
 \left\{
    \begin{array}{ll}
        {\eta} - 1 +\frac{\log\pa{1-\frac{\frac{1-u}{u}(a-\eta)}{b}}}{\log(1/u)} & \mbox{if } u \in \pa{\frac{{a-\eta}}{{a+b-\eta}},1} \\
        -\infty & \mbox{otherwise.}
    \end{array}
\right.
\] Then, $f(\eta,\cdot)$ is non-decreasing on $\pa{\frac{{a-\eta}}{{a+b-\eta}},1}$. Indeed, we have
\[\frac{\partial f(\eta,u)}{\partial u} = \frac{\log\pa{1-\frac{\frac{1-u}{u}(a-\eta)}{b}} + \frac{\pa{a-\eta}\log(1/u)}{\pa{a+b-\eta}{}u - \pa{a-\eta}}}{u\log^2(1/u)},\]
which is is non-negative because the numerator is a sum of two non-increasing functions, both of which tend to \( 0 \) as \( u \to 1^- \). From
\[S_\infty(a,b,\cdot) = \sup{\sset{\eta\in[0,a],~f(\eta,\cdot) < 0}},\]
we get that this function is non-increasing. Indeed, a feasible $\eta$ for $u_1\in (0,1)$ is also feasible for $u_2\in (0,u_1)$, so that the feasible sets are nested. 
\end{itemize}
\end{proof}

\begin{proposition}\label{prop:sup_PI}
Let $\Pi$ be a the set of finite measurable partition of $\Omega$. Then, for $t>0$, $\nu\in \cM_1(\Omega)$, $f\in\cC\pa{\Omega}$ and $u\in [f_{\min},f_{\max})$, where $f_{\max}\triangleq \max_{x\in \Omega} f(x)$, $f_{\min}\triangleq \min_{x\in \Omega} f(x)$,
\[\exp\pa{-t\Kinf\pa{\nu,u,f}} = \sup_{(A_1,\dots, A_k)\in \Pi} \EEs{X\sim \DP\pa{t\nu}}{\exp\pa{-t \Kinf(\pa{\nu(A_i)}_i,u,\pa{\EEccs{X}{f}{A_i}}_i)}}.\]
In particular, for $t\to\infty$, we recover that
\[\Kinf\pa{\nu,u,f} = \inf_{(A_1,\dots, A_k)\in \Pi}  \Kinf(\pa{\nu(A_i)}_i,u,\pa{\EEccs{\nu}{f}{A_i}}_i).\]
\end{proposition}
\begin{proof}
Let $A_1\sqcup~\dots~\sqcup A_k=\Omega$ be a finite measurable partition of $\Omega$. 
Let $\ell(p)\triangleq \Kinf(\pa{\nu(A_i)}_i,u,\pa{\EEccs{p}{f}{A_i}}_i)$ for $p\in \cM_1(\Omega)$, then
    \begin{align*}
\ell(p)&={\max_{\lambda \in [0,{1}/{(\max_{i\in [k]}{\EEccs{p}{f}{A_i}}-u)}]} \sum_{i\in [k]}\nu(A_i)\log(1-\lambda(\EEccs{p}{f}{A_i}-u))}
    \\&=
    \max_{\lambda \in [0,{1}/{(\max_{i\in [k]}{\EEccs{p}{f}{A_i}}-u)}]} -t^{-1}\log\EEs{G\sim \cG\pa{t\nu}}{e^{\lambda\sum_{i\in [k]}G(A_i)(\EEccs{p}{f}{A_i}-u)}}
    \\&\geq \max_{\lambda \in [0,1/\pa{f_{\max}-u}]} -t^{-1}\log\EEs{G\sim \cG\pa{t\nu}}{e^{\lambda\sum_{i\in [k]}G(A_i)(\EEccs{p}{f}{A_i}-u)}},
    \end{align*}
so
    \begin{align*}
    &\EEs{X\sim \DP\pa{t\nu}}{\exp\pa{-t \ell(X)}} \\&\leq
    \EEs{X\sim \DP\pa{t\nu}}{\min_{\lambda \in [0,1/\pa{f_{\max}-u}]} \EEccs{G\sim \cG\pa{t\nu}}{e^{\lambda\sum_{i\in [k]}G(A_i)(\EEccs{X}{f}{A_i}-u)}}{X}}
    \\&\leq
    \min_{\lambda \in [0,1/\pa{f_{\max}-u}]}\EEs{X\sim \DP\pa{t\nu}}{ \EEccs{G\sim \cG\pa{t\nu}}{e^{\lambda\sum_{i\in [k]}G(A_i)(\EEccs{X}{f}{A_i}-u)}}{X}}
\\&=
    \min_{\lambda \in [0,1/\pa{f_{\max}-u}]}{ \EEs{G\sim \cG\pa{t\nu}}{e^{\lambda\pa{\int (f-u)dG}}}} = \exp\pa{-t\Kinf\pa{\nu,u,f}}.
\end{align*}
Since this is true for any finite partition, we have
\[\sup_{A_1\sqcup\dots\sqcup A_k\in \Pi} \EEs{X\sim \DP\pa{t\nu}}{\exp\pa{-t \Kinf(\pa{\nu(A_i)}_i,u,\pa{\EEccs{X}{f}{A_i}}_i)}} \leq \exp\pa{-t\Kinf\pa{\nu,u,f}}.\]
To have the equality, consider some $\varepsilon>0$. Then there exists $A_1\sqcup\dots\sqcup A_k\in \Pi$ such that for any $p\in \cM_1(\Omega)$, $\norm{\sum_{i\in [k]}\EEccs{p}{f}{A_i}\II{\cdot\in A_i} - f}_\infty \leq\varepsilon$. We thus have
\begin{align*}\exp\pa{-t\Kinf\pa{\nu,u - \varepsilon,f}} &= \exp\pa{-t\Kinf\pa{\nu,u,f+\varepsilon}}\\&\leq \EEs{X\sim \DP(t\nu)}{\exp\pa{-t\Kinf(\pa{\nu(A_i)}_i,u,\pa{\EEccs{p}{f}{A_i}}_i)}}.\end{align*}
We get our result using the left-continuity of $\Kinf$ \cite{garivier2022klucbswitch}.
\end{proof}

\begin{proposition}
    \label{prop:Ee^Z}
    Let $Z$ be a non-negative random variable. Then, for $s,t>0$, we have $\EE{Z^{s+t}} \geq \EE{Z^{s}}\EE{Z^{t}}$.
\end{proposition}
\begin{proof}Let $s,t>0$. Since both $x\mapsto x^\frac{s}{s+t}$ and $x\mapsto x^\frac{t}{s+t}$ are concave on $\R_+$, we have from Jensen's inequality
\[\EE{Z^{s}} = \EE{Z^{(s+t)\frac{s}{s+t}}} \leq \EE{Z^{s+t}}^\frac{s}{s+t},\]
\[\EE{Z^{t}} = \EE{Z^{(s+t)\frac{t}{s+t}}} \leq \EE{Z^{s+t}}^\frac{t}{s+t}.\]
We obtain the result by multiplying these inequalities.
\end{proof}

\begin{proposition}\label{prop:quicka_0,a_1,b_0,b_1}
    Let $a_0,a_1,b_0,b_1\in \R_+$. Then
    \[\pa{a_0a_1 - b_0b_1}^+\geq \pa{a_0 - b_0}^+\pa{a_1 - b_1}^+.\]
\end{proposition}
\begin{proof}
    We assume that $a_0 > b_0$ and $a_1 > b_1$, as the result is trivial otherwise. We have
    \begin{align*}
        0&<(a_0-b_0)b_1 + (a_1-b_1) b_0
        = a_0a_1 - b_0b_1 - (a_0 - b_0)(a_1 - b_1).
    \end{align*}
\end{proof}

\section{LDP in \(\R^d\)}
\label{app:ldp_dirichlet}
\subsection{LDP for the Gamma distribution}
\begin{fact}[CGF of the gamma distribution]\label{fact:camulant_gamma}
Let $Z\sim\text{Gamma}(k,1)$ be a sample from the gamma distribution with
shape parameter $k>0$ and scale parameter $1$. We have
\[\log\EE{e^{\lambda Z}}=\left\{
    \begin{array}{ll}
        -k\log\pa{1-\lambda} & \mbox{if } \lambda < 1 \\
        \infty & \mbox{otherwise.}
    \end{array}
\right.\]
\end{fact}

\begin{fact}[Gärtner-Ellis theorem \cite{gartner1977large,ellis1984large}]
\label{fact:ellis}Let $\pa{\mu_n}$ be a sequence of probability
measures on $\pa{\R^d,\cB(\R^d)}$. Assume that for all $t\in \R^d$, the limit 
\[c(t)\triangleq\lim_{n\to\infty}\frac{1}{n}\log\int_{\R^d}e^{\sca{ nt,x}}d\mu_n(x)\in [-\infty,+\infty]\]
exists and that $0\in D_{c}^o$, where $D_c\triangleq \sset{x\in \R^d,~c(x)<\infty}$. We also define the Fenchel-Legendre transform as
\[c^*(x)\triangleq \sup_{t\in \R^d}\pa{\sca{t,x} - c(t)}.\]
Then we have:
\begin{itemize}
    \item[$(i)$] 
    $\limsup_{n\rightarrow\infty}\frac{1}{n}\log \mu_n(F)\leq -\inf_{x\in F}c^*(x),~~\forall F\subset\mathbb{R}^d~ \text{ closed.}$
    \item[$(ii)$] Assume in addition that $c$ is lower semi-continuous on $\R^d$, differentiable on $D_c^o$, and either $D_c=\R^d$ or $c$ is steep, i.e., for all $y \in \partial D^o_c$,
\[\lim_{x\to y}\norm{\nabla c(x)}=\infty.\]
Then,    $\liminf_{n\rightarrow\infty}\frac{1}{n}\log \mu_n(G)\geq -\inf_{x\in G}c^*(x),~~ \forall G\subset\mathbb{R}^d$ open.
This means $\pa{\mu_n}$ satisfies an
LDP with rate function $c^*$.
\end{itemize}
\end{fact}

\begin{lemma}[LDP for the gamma distribution]\label{lem:LDP_gamma}
Let $Z_n\sim\text{Gamma}(k_n,1)$, with $(k_n)\in  \pa{\R_+^*}^\N$ such that
\[\frac{k_n}{n}\to k\in \R_+.\]
Then, $\pa{Z_n/n}$  satisfies point $(i)$ of an LDP in $\R$ with the good rate function 
\[I(x)=\left\{
    \begin{array}{ll}
        x-k +k\log\pa{\frac{k}{x}} & \mbox{if } x>0 \\
         0 & \mbox{if }x=k=0 \\
        \infty & \mbox{otherwise.}
    \end{array}
\right.\]
Furthermore, if \[\frac{\log(\min\pa{k_n, 1})}{n}\to 0,\] 
then $\pa{Z_n/n}$ also satisfies point $(ii)$.
\end{lemma}
\begin{proof}
We can explicitly express the CGF of $Z_n$ using Fact~\ref{fact:camulant_gamma}. Then, for $\mu_n=\PP{Z_n/n\in\cdot}$, we see that the limit condition in Fact~\ref{fact:ellis} is satisfied: For all $t<1$,
\[c(t)=\lim_{n\to\infty}\frac{-k_n}{n}\log\pa{1-t} = -k\log(1-t).\]
See also that the convex conjugate of the function $c$ is $I$, i.e., $c^*=I$. For $k>0$, $c$ also satisfies the conditions of $(ii)$ in Fact~\ref{fact:ellis}, but for $k=0$, the steepness condition is not satisfied.
From here, we already have that point $(i)$ of an LDP is satisfied for all $k\geq 0$, and that point $(ii)$ is satisfied for $k>0$.
For $k=0$, we have to treat the lower bound separately: We have $I(x)=x+\infty\II{x< 0}$. Consider $G\subset \R$, open. 
If $G\cap \R_+^* = \emptyset$, then $G\subset (-\infty,0)$ and $I=\infty$ on~$G$, so $-\inf_{x\in G}I(x)=-\infty$ and the desired lower bound is trivial. Else, we clearly see that
$\inf_{x\in G} I(x)=\inf_{x\in G\cap \R_+^*} I(x)\geq 0$. 
Let $x$ be in the open set $G\cap \R_+^*$. Then, there is a sub-interval of the form $(x-r,x+r)\subset G\cap \R_+^*$, with $r > 0$. We have
    \begin{align*}
    \PP{Z_n/n\in G}&\geq \PP{Z_n/n\in G\cap \R_+^*} 
    \\&\geq 
    \PP{Z_n/n\in (x,x+r)}
    \\&=
    {\frac{1}{\Gamma(k_n)}\int_{nx}^{n(x+r)}e^{-u}u^{k_n-1}du}
    \\&\geq
    {\frac{(nx)^{k_n-1}}{\Gamma(k_n)}\int_{nx}^{n(x+r)}e^{-u}du}
    \\&=
    {\frac{(nx)^{k_n-1}}{\Gamma(k_n)}e^{-nx}\pa{1-e^{-nr}}}
    \\&\geq
    {\frac{(nx)^{k_n-1}}{\max (k_n^{-1}, k_n^{k_n-1})}e^{-nx}\pa{1-e^{-nr}}}
    \\&=
    {{\min (k_n, 1)^{k_n}}\frac{k_n}{nx}\pa{\frac{nx}{k_n}}^{k_n}}{}e^{-nx}\pa{1-e^{-nr}}
    \\&\geq 
        {\frac{{\min (k_n, 1)}^2}{nx}\pa{\frac{nx}{k_n}}^{k_n}}{}e^{-nx}\pa{1-e^{-nr}}
    .
    \end{align*}
    Now, looking at $\frac{1}{n}\log\PP{Z_n/n\in G}$, we get the lower bound
    \[-x + \underbrace{\frac{\log\pa{1-e^{-nr}}}{n}}_{\to 0} + \underbrace{\frac{k_n}{n}\log\pa{\frac{nx}{k_n}}}_{\to -k\log(k/x)= 0} 
    +
    \underbrace{\frac{1}{n}\log\pa{\frac{1}{nx}}}_{\to 0}
    +
    \underbrace{\frac{2\log\pa{{\min (k_n, 1)}}}{n}}_{\to 0},\]
    i.e., $\liminf_{n\to\infty}\frac{1}{n}\log\PP{Z_n/n\in G}\geq -x = -I(x)$. We get the desired result by taking the $\sup$ on $x\in G\cap \R_+^*$. 
Notice that we used here the assumption ${\log(\min (k_n, 1))}/{n}\to 0$, i.e., that $k_n$ does not decay exponentially in $n$.
\end{proof}

\subsection{LDP for the Dirichlet distribution}

\begin{fact}[Joint LDP in $\R^d$ \cite{chaganty1997large}]\label{fact:joint}
Let $(Y_{1,n},\dots,Y_{d,n})$ be independent random variable in $\R$. If for each $i$, $(Y_{i,n})$ satisfies point $(i)$ (respectively point $(ii)$) of an LDP with good rate function $I_i$, then  $\pa{(Y_{i,n})_i}\in \pa{\R^d}^\N$ satisfies point $(i)$ (respectively point $(ii)$) of an LDP with good rate function $\sum_{i\in [d]} I_i$.
\end{fact}

\begin{fact}[Contraction principle \cite{dembo2009large}]\label{fact:contraction}  
If $(\mu_n)$ satisfies point $(i)$ (respectively point $(ii)$) of an LDP with good rate $I$, then $\pa{\mu_n\pa{T^{-1}(\cdot)}}$, where $T$ is continuous, satisfies point $(i)$ (respectively point $(ii)$) of an LDP with good rate
\[J(y)=\inf_{x \in T^{-1}(\sset{y})}I(x).\]
\end{fact}

\begin{fact}[Varadhan’s integral lemma \cite{Varadhan1966asymptotic}]
\label{fact:varadhan}
Let $(Z_n)\in \pa{\R^d}^\N$ be sequence of random vectors  satisfying an LDP with a
good rate function $I:\R^d\to [0,\infty]$ and let $\phi\in \cC(\R^d)$. Assume either the tail condition
\[\lim_{M\to\infty}\limsup_n \frac{1}{n}\log\EE{e^{n\phi(Z_n)}\II{\phi(Z_n)\geq M}} = -\infty\]
or the following moment condition for some $\gamma>1$,
\[\limsup_n \frac{1}{n}\log\EE{e^{n\gamma\phi(Z_n)}}<\infty.\]
Then
\[\lim_{n\to\infty}\frac{1}{n}\log\EE{e^{n\phi(Z_n)}} = \sup_{x\in \R^d}\pa{\phi(x) - I(x)}.\]
\end{fact}

\begin{proposition}\label{prop:Kl-gamma} Let $x\in \cM_1\pa{[d]}$ and for all $i\in [d],$ let
    \[f_i(t)=\left\{
    \begin{array}{ll}
\infty & \mbox{if } t< 0~\text{or}~t=0<x_i\\
        0 & \mbox{if } t=0=x_i\\
        t-x_i+x_i\log\pa{{x_i}/{t}} & \mbox{if } t>0.\\
    \end{array}
\right.\]
We define
    \[ I:y\in \R^d\mapsto \inf \sset{\sum_{i\in [d]} f_i(z_i),~\frac{z_i}{\sum_{j\in [d]}z_j}=y_i~\forall i\in [d]}.\]
    Then, we have
    \[I(y)=\left\{
    \begin{array}{ll}
        \KL{x}{y} & \mbox{if } y \in  \cM_1\pa{[d]}\\
        \infty & \mbox{otherwise.}
    \end{array}
\right.\]
\end{proposition}
\begin{proof}
If there is some $j$ such that $y_j<0$, then any $z$ included in the infimum must have some $z_i<0$, so $I(y) = \infty$. If $\sum_{j\in [d]} y_j \neq 1$, then there is no $z\in \R^d$ such that $y_i=z_i/\sum_{j\in [d]} z_j$ for all $i\in [d]$ and $I(y)=\inf \emptyset = \infty$. Hence, we consider $y\in \cM_1\pa{[d]}$. If there is some $j$ such that $y_j=0<x_j$, then $\KL{x}{y}=\infty$ by definition. On the other hand, any $z$ included in the infimum is such that $z_j=0$, meaning that $I(y) = \infty$.
Now, if $y$ is such that for all $i\in [d]$, we have $y_i=0 \Rightarrow x_i=0$, then
\[I(y)=\inf \sset{\sum_{i\in [d],~y_i>0} f_i(z_i),~\frac{z_i}{\sum_{j\in [d]}z_j}=y_i~\forall i\in [d]}.\]
The constraint is equivalent to $z=\beta y$, for $\beta>0$. Setting the derivative of \[\beta\mapsto\sum_{i\in [d],~y_i>0} f_i(\beta y_i)\] equal to $0$ gives
\[0=\sum_{i\in [d],~y_i>0}\pa{y_i-\frac{x_i}{\beta}}=1-\frac{1}{\beta},\]
where the last equality uses that $x,y\in \cM_1\pa{[d]}$. Thus, the infimum is achieved at $\beta=1$ as each $f_i$ is convex. This gives
\[I(y)= \sum_{i\in [d],~y_i>0} f_i(y_i) = \sum_{i\in [d],~y_i>0} x_i\log(x_i/y_i) = \KL{x}{y}.\]
\end{proof}

\begin{theorem}[LDP for the Dirichlet distribution]\label{thm:ldp_dir}
    Let $(k_n) = ((k_{i,n})_{i\in [d]})\in \pa{{\R_+^*}^d}^\N$ such that
\[\frac{k_n}{n}\to x\in \cM_1\pa{[d]},\]
and let $Y_n\sim\text{Dir}\pa{k_{1,n},\dots,k_{d,n}}$.
Then, $(Y_n)$ satisfies point $(i)$ of an LDP in $\R^d$ with the good rate function
\[I(y)=\left\{
    \begin{array}{ll}
        \KL{x}{y} & \mbox{if } y \in  \cM_1\pa{[d]}\\
        \infty & \mbox{otherwise.}
    \end{array}
\right.\]
Furthermore, if for all $i\in [d]$ \[\frac{\log(\min\pa{k_{i,n}, 1})}{n}\to 0,\] 
then $(Y_n)$ also satisfies point $(ii)$.
\end{theorem}
\begin{proof}
    We have 
    \[Y_n\sim h(Z_n)\triangleq\pa{\frac{Z_{1,n}}{\sum_{i\in [d]} Z_{j,n}},\dots,\frac{Z_{d,n}}{\sum_{i\in [d]} Z_{j,n}}},\]
    where the $Z_{j,n}\sim \text{Gamma}(k_{j,n},1)$ are independent gamma random variables. 
    From Lemma~\ref{lem:LDP_gamma}, each $(Z_{i,n}/n)$ satisfies point $(i)$ of an LDP with the good rate function
    \[f_i(t)=\left\{
    \begin{array}{ll}
\infty & \mbox{if } t< 0~\text{or}~t=0<x_i\\
        0 & \mbox{if } t=0=x_i\\
        t-x_i+x_i\log\pa{{x_i}/{t}} & \mbox{if } t>0,\\
    \end{array}
\right.\]
and point $(ii)$ whenever ${\log(\min\pa{k_{i,n},1})}/{n}\to 0.$ 
From Fact~\ref{fact:joint}, we have that $(Z_n/n)=((Z_{i,n}/n)_{i\in [d]})$ satisfies point $(i)$ of an LDP in $\R^d$ with good rate $\sum_{j\in [d]} f_j$, and point $(ii)$ whenever ${\log(\min\pa{k_{i,n},1})}/{n}\to 0$ $\forall~i\in[d]$. Since the map $h$ is continuous, we get from Fact~\ref{fact:contraction} that $(Y_n)$ satisfies point $(i)$ of an LDP with good rate function
\[ I(y)=\inf \sset{\sum_{i\in [d]} f_i(z_i),~\frac{z_i}{\sum_{j\in [d]}z_j}=y_i~\forall i\in [d]},\]
and point $(ii)$ whenever ${\log(\min\pa{k_{i,n},1})}/{n}\to 0~~~\forall~i\in[d]$.
From Proposition~\ref{prop:Kl-gamma}, we get \[I(y)=\left\{
    \begin{array}{ll}
        \KL{x}{y} & \mbox{if } y \in  \cM_1\pa{[d]}\\
        \infty & \mbox{otherwise,}
    \end{array}
\right.\]
which finishes the proof.
\end{proof}

\begin{corollary}[CGF limit for the Dirichlet distribution]\label{cor:varadhan_dir}
Let $(k_n) = ((k_{i,n})_{i\in [d]})\in \pa{{\R_+^*}^d}^\N$ such that
\[\frac{k_n}{n}\to x\in \cM_1([d])~\text{ and }~\frac{\log(\min\pa{k_{i,n}, 1})}{n}\to 0~~\forall~i\in [d].\]
Let $Y_n\sim\text{Dir}\pa{k_{1,n},\dots,k_{d,n}}$. Then, for all $\lambda\in \R^d$, \[\lim_{n\to\infty}\frac{1}{n}\log\EE{e^{n\sca{\lambda,Y_n}}} = \sup_{y\in \cM_1\pa{[d]}}\pa{\sca{\lambda,y} - \KL{x}{y}}.\]
\end{corollary}
\begin{proof}
    Since $\sca{\lambda,Y_n}\leq \max_{i\in [d]}\lambda_i$, we can use Fact~\ref{fact:varadhan} with
    $\phi = \sca{\lambda,\cdot}$, considering the LDP obtained from Theorem~\ref{thm:ldp_dir}.
\end{proof}

\section{LDP for DPs}
\label{app:ldp}

\subsection{CGF limit for DPs}

\begin{fact}[Continuity of convex functions \cite{rockafellar1974conjugate}]\label{fact:contconv}
Let $U$ be a Banach space. Then, if $c:U\to\R$ is convex and bounded on a neighborhood of some point of $U$, then it is continuous on $D_{c}^o$, where $D_c\triangleq \sset{x\in U,~c(x)<\infty}$.
\end{fact}

\begin{fact}[Characterization of the KL divergence \cite{georgii2011gibbs}]
\label{fact:chara_kl}
Let $(A_{1,n},\dots,A_{k_n,n})$ be a sequence of measurable partitions of $\Omega$
such that the corresponding $\sigma$-algebras, increase to $\cB(\Omega)$. Then, for all $\mu,\nu\in\cM_1\pa{\Omega}$, $\KL{\mu}{\nu}$ is equal to
\[\sup_n \KL{\pa{\mu(A_{i,n})}_{i\in [k_n]}}{\pa{\nu(A_{i,n})}_{i\in [k_n]}}= \lim_n \KL{\pa{\mu(A_{i,n})}_{i\in [k_n]}}{\pa{\nu(A_{i,n})}_{i\in [k_n]}}.\]
\end{fact}

\begin{fact}[LDP general upper bound \cite{dembo2009large}]\label{fact:LDP_general_upper_bound}
    Let $\cX$ be a Hausdorff (real) topological vector space. Let $(\mu_n)\in \cM_1(\cX)^\N$. Let 
    \[c(\lambda)\triangleq\limsup_n\frac{1}{n}\log\int_{\cX}e^{n\lambda(x)}d\mu_n(x),\quad \lambda\in \cX^*.\]
    We also define the Fenchel-Legendre transform as
\[c^*(x)\triangleq \sup_{\lambda\in \cX^*}\pa{\sca{\lambda,x} - c(x)},\quad x\in \cX.\]
Then, for any compact set $K\subset \cX$, we have
\[\limsup_{n}\frac{1}{n}\log\mu_n(K)\leq -\inf_{x\in K}c^*(x).\]
\end{fact}

\begin{proposition}\label{prop:KL=KL_k}
Let $(A_1,\dots,A_k)$ be a measurable partition of $\Omega$.   
Then, for all $\phi:\Omega\to\R$ bounded and measurable with respect to $\sigma\pa{A_1,\dots,A_k}$,
and all $\mu\in \cM_1(\Omega)$, we have
    \begin{align*}\sup_{\nu\in \cM_1(\Omega)}\pa{\int \phi d\nu - \KL{\mu}{\nu}} = \sup_{\nu\in \cM_1(\Omega)}\pa{\int \phi d\nu - \KL{\pa{\mu(A_{i})}_{i\in [k]}}{\pa{\nu(A_{i})}_{i\in [k]}}}.\end{align*}
\end{proposition}
\begin{proof}
Let $\mu\in \cM_1(\Omega)$. If $\nu\in \cM_1(\Omega)$ is such that $\mu(A_i)>0=\nu(A_i)$ for some $A_i$, then 
\[\KL{\pa{\mu(A_{i})}_{i\in [k]}}{\pa{\nu(A_{i})}_{i\in [k]}}
    =
    \KL{\mu}{\nu}=\infty,\]
so such $\nu$ can be excluded in both suprema. Let's thus consider $\nu$ such that for all $i\in[k]$, $\mu(A_i)>0\Rightarrow \nu(A_i)>0$. We first consider the $\geq$ in the equality of Proposition~\ref{prop:KL=KL_k}.
We define $\delta\in \cM_1(\Omega)$ as 
\[\delta\triangleq\sum_{i\in [k],~\mu(A_i)>0} \frac{\nu(A_i)}{\mu(A_i)}\mu(\cdot\cap A_i) + \sum_{i\in [k],~\mu(A_i)=0}\nu(\cdot\cap A_i).\]
Then, $\mu$ is absolutely continuous with respect to $\delta$ and 
\begin{align*}\KL{\mu}{\delta}&=\int \log\pa{\frac{d\mu}{d\delta}}d\mu\\&= \sum_{i\in [k],~\mu(A_i)>0} \int_{A_i}\log\pa{\frac{d\mu}{d\delta}}d\mu\\& = \sum_{i\in [k],~\mu(A_i)>0} \mu(A_i)\log\pa{\frac{\mu(A_i)}{\nu(A_i)}}\\&=
\KL{\pa{\mu(A_{i})}_{i\in [k]}}{\pa{\nu(A_{i})}_{i\in [k]}}.
\end{align*}
In addition, since for all $i\in [k]$, $\nu(A_i)=\delta(A_i)$, we have \[\int \phi d\nu = \int \phi d\delta.\]
Since $\nu$ was arbitrary, we deduce that
\[\sup_{\nu\in \cM_1(\Omega)}\pa{\int \phi d\nu - \KL{\mu}{\nu}} \geq  \sup_{\nu\in \cM_1(\Omega)}\pa{\int \phi d\nu - \KL{\pa{\mu(A_{i})}_{i\in [k]}}{\pa{\nu(A_{i})}_{i\in [k]}}}.\]
For the reversed inequality, it is sufficient to prove that 
\[\KL{\pa{\mu(A_{i})}_{i\in [k]}}{\pa{\nu(A_{i})}_{i\in [k]}}
    \leq \KL{\mu}{\nu}.\]
    We can assume that $\mu \ll \nu$, as otherwise $\KL{\mu}{\nu}=\infty$ and the inequality is trivial.
    By convexity of $f:x\mapsto x\log(x)$ on $[0,\infty)$, we have
    \begin{align*}\KL{\pa{\mu(A_{i})}_{i\in [k]}}{\pa{\nu(A_{i})}_{i\in [k]}}&=
    \sum_{i\in [k],~\nu(A_i)>0} \nu(A_i)f\pa{\frac{1}{\nu(A_i)}\int_{A_i}\frac{d\mu}{d\nu}d\nu}
    \\&\leq 
    \sum_{i\in [k],~\nu(A_i)>0} \nu(A_i){\frac{1}{\nu(A_i)}\int_{A_i}f\pa{\frac{d\mu}{d\nu}}d\nu}
    \\&= 
\sum_{i\in [k]} \int_{A_i}f\pa{\frac{d\mu}{d\nu}}d\nu
     = \KL{\mu}{\nu}.\end{align*}
\end{proof}

\begin{lemma}[CGF limit for DPs, the case of simple functions]\label{lem:mu-conti}
Let $(\mu_n)$ be a sequence of measures in $\cM(\Omega)$, with support $\Omega$, such that \[\frac{\mu_n}{n}\to\mu\in \cM_1\pa{\Omega},~~\text{weakly.}\] 
Let $(A_1,\dots,A_k)$ be a measurable partition of $\Omega$. Suppose that
each $A_i$ is a $\mu$-continuity set with non-empty interior and that \({\log\pa{\min\pa{\mu_n(A_i),1}}}=o\pa{n}\). Then, for all $\phi$  bounded and measurable with respect to $\sigma\pa{A_1,\dots,A_k}$,
\[\lim_{n\to\infty} 
\frac{1}{n}\log\int_{\cX}e^{n\phi(x)}d\mu_n(x)
 = \sup_{\nu\in \cM_1(\Omega)}\pa{\int \phi d\nu - \KL{\mu}{\nu}}.\]
\end{lemma}

\begin{proof} Let $(A_1,\dots,A_k)$ be a measurable partition of $\Omega$ such that
each $A_i$ is a $\mu$-continuity set with non-empty interior. Let $\phi$ be bounded and measurable with respect to the $\sigma$-algebra
generated by this partition. Then we can write $\phi=\sum_{i\in [k]}\lambda_i\II{\cdot\in A_i},$
where $\lambda=(\lambda_i)_{i\in [k]}\in \R^k$. By the assumption that $\mu_n$ has support $\Omega$, and that each $A_i$ is a $\mu$-continuity set with non-empty interior, we have that the sequence $\pa{\pa{\mu_n(A_i)}_{i\in [k]}}\in \pa{\R^k}^\N$
satisfies the assumptions of Corollary~\ref{cor:varadhan_dir}, so that
\begin{align*}\lim_{n\to\infty} \frac{1}{n}\log\int_{\cX}e^{n\phi(x)}d\mu_n(x)
  &= \sup_{\nu\in \cM_1(\Omega)}\pa{\int \phi d\nu - \KL{\pa{\mu(A_{i})}_{i\in [k]}}{\pa{\nu(A_{i})}_{i\in [k]}}}
\\&=
\sup_{\nu\in \cM_1(\Omega)}\pa{\int \phi d\nu - \KL{\mu}{\nu}} 
,\end{align*}
where the last equality is from Proposition~\ref{prop:KL=KL_k}.
\end{proof}

\begin{lemma}[CGF limit for DPs]
\label{lem:lynch_DP}
Let $(\mu_n)$ be a sequence of measures in $\cM(\Omega)$, with support $\Omega$, such that \[\frac{\mu_n}{n}\to\mu\in \cM_1\pa{\Omega},~~\text{weakly.}\]
If
\({\log\pa{\min\pa{\mu_n(A),1}}}=o\pa{n}\) for all $\mu$-continuity set $A$ with non-empty interior, 
    then, for all continuous functions $\phi:\Omega\to\R$, the limit
    \[\Lambda(\phi)\triangleq\lim_{n\to\infty} \frac{1}{n}\Lambda_{\mu_n}(n\phi)=\lim_{n\to\infty} 
\frac{1}{n}\log\int_{\cX}e^{n\phi(x)}d\mu_n(x)
 \]
    exists and is finite. $\Lambda : \cC(\Omega)\to \R$ is convex and continuous, and
    \[\Lambda(\phi)=\sup_{\nu\in \cM_1(\Omega)}\pa{\int \phi d\nu - \KL{\mu}{\nu}}.\]
\end{lemma}
\begin{proof}
    Let $\phi\in \cC(\Omega)$ and $\epsilon>0$.
There exists a simple function $g=\sum_{i\in [k]} c_i \II{\cdot\in A_i}$, where $(A_1,\dots,A_k)$ is a measurable partition of $\Omega$, such that $\norm{\phi-g}_\infty<\epsilon$.
    Since $\phi$ is continuous, we can choose the $A_i$ to be $\mu$-continuity sets with non-empty interiors, such that from Lemma~\ref{lem:mu-conti}, \[\Lambda(g)\triangleq \lim_{n\to\infty} \frac{1}{n}\Lambda_{\mu_n}(n g) = \sup_{\nu\in \cM_1(\Omega)}\pa{\int g d\nu - \KL{\mu}{\nu}}.\]
    We see that $-n\epsilon + \Lambda_{\mu_n}(n g)\leq \Lambda_{\mu_n}(n\phi)\leq n\epsilon + \Lambda_{\mu_n}(n g)$, so \[-\epsilon + \Lambda(g)\leq \liminf_{n\to\infty} \frac{1}{n}\Lambda_{\mu_n}(n\phi)\leq \limsup_{n\to\infty} \frac{1}{n}\Lambda_{\mu_n}(n\phi)\leq \epsilon + \Lambda(g),\]
    and so
    \[0\leq \limsup_{n\to\infty} \frac{1}{n}\Lambda_{\mu_n}(n\phi) - \liminf_{n\to\infty} \frac{1}{n}\Lambda_{\mu_n}(n\phi)\leq 2\epsilon.\]
    Since $\epsilon$ is arbitrary, it
follows that \[\Lambda(\phi)\triangleq\lim_{n\to\infty} \frac{1}{n}\Lambda_{\mu_n}(n\phi)\]
    exists and is finite for all $\phi\in \cC(\Omega)$ (such function $\phi$ is thus bounded as $\Omega$ is compact). For all $\phi,\phi'\in \cC(\Omega)$, we have
    \begin{align*}\Lambda(\phi)-\Lambda(\phi') &= \lim_{n\to\infty} \frac{1}{n}\pa{\Lambda_{\mu_n}(n\phi)-\Lambda_{\mu_n}(n\phi')}\\&\leq 
    \lim_{n\to\infty} \frac{1}{n}\pa{\Lambda_{\mu_n}(n\phi' + n\norm{\phi-\phi'}_\infty)-\Lambda_{\mu_n}(n\phi')}
    \\&= \norm{\phi-\phi'}_\infty,
    \end{align*}
    so $\Lambda$ is continuous.
    For $\phi\in L^\infty(\Omega)$, we define \[\text{KL}^*(\phi)\triangleq \sup_{\nu\in \cM_1(\Omega)}\pa{\int \phi d\nu - \KL{\mu}{\nu}}.\]
    we have
    \[\abs{\text{KL}^*(\phi)}\leq \abs{\sup_{\nu\in \cM_1(\Omega)}{\int {\phi} d\nu - \underbrace{\inf_{\nu\in \cM_1(\Omega)} \KL{\mu}{\nu}}_{=0}}}\leq \norm{\phi}_\infty,\]
    so, from Fact~\ref{fact:contconv}, $\text{KL}^*$ is continuous on the interior of its domain, i.e., on $L^\infty(\Omega)$. In particular, it is continuous on $\cC\pa{\Omega}$. By Lemma~\ref{lem:mu-conti}, $\Lambda=\text{KL}^*$ on the set of simple functions of the form $g=\sum_{i\in [k]} c_i \II{\cdot\in A_i}$ such that 
    $(A_1,\dots,A_k)$ is a measurable partition of $\Omega$ and
    each $A_i$ is a $\mu$-continuity sets with non-empty interiors.
    Since such functions are dense in $\cC\pa{\Omega}$ and both $\Lambda, \text{KL}^*$ are continuous on $\cC\pa{\Omega}$, it follows that $\Lambda=\text{KL}^*$ on $\cC\pa{\Omega}$. In particular, $\Lambda$ is convex. 
\end{proof}

\begin{corollary}
    Let $\phi\in \cC(\Omega)$. Assume that $\mu\in \cM_1(\Omega)$ has support $\Omega$ and that $\delta\in \cM(\Omega)$. Then, 
    \[\lim_{n\to\infty} \frac{1}{n}\Lambda_{\delta+n\mu}(n\phi)=\sup_{\nu\in \cM_1(\Omega)}\pa{\int \phi d\nu - \KL{\mu}{\nu}}.\]
\end{corollary}

\subsection{LDP for DPs: Proof of Theorem~\ref{thm:gen_ldp_DP}}
\begin{proof}[Proof of Theorem~\ref{thm:gen_ldp_DP}]
 We have from Lemma~\ref{lem:lynch_DP} that $\Lambda(\cdot)\triangleq \lim_{n\to\infty} \frac{1}{n}\Lambda_{\mu_n}(n~\cdot)$ is the convex
conjugate of $\KL{\mu}{\cdot}$, which is convex, and lower semicontinuous in the
weak topology. Thus, the two functions 
$\Lambda(\cdot)$ and $\KL{\mu}{\cdot}$ are convex duals of each other (recall that
$\cM_1\pa{\Omega}$ is Polish since $\Omega$ is Polish). From Fact~\ref{fact:LDP_general_upper_bound}, we thus get the large deviations upper bound for compact
subsets of $\cM_1(\Omega)$.  But $\Omega$ was assumed to be compact, hence $\cM_1(\Omega)$ is compact in the
weak topology, so the upper bound holds for all closed sets in $\cM_1(\Omega)$. We
now turn to the proof of the lower bound.  
The weak topology on $\cM_1(\Omega)$ is generated by the sets
\[U_{\phi,x,\eta} = \sset{\nu\in \cM_1(\Omega),~\abs{\int_\Omega \phi d\nu - x }<\eta},\quad\phi\in \cC\pa{\Omega},~x\in \R,~\eta>0.\]
Given such a set and some $\epsilon\in (0,\eta/3),$ there exists a sequence of measurable partitions of $\Omega$, denoted $(A_k)_k=\pa{(A_{1,k},\dots,A_{n_k,k})}_k$, and
a sequence of simple $\sigma(A_k)$-measurable functions $\phi_k = \sum_{i\in [n_k]} c_{i,k}\II{\cdot\in A_{i,k}}$  
such that
\begin{itemize}
    \item $\sigma(A_k)$ increases to $\cB(\Omega)$,
    \item For all $k$ and $i\in [n_k]$, $A_{i,k}$ is a $\mu$-continuity set with non-empty interior,
    \item There is some $k_0$ such that for $k>k_0$, $\norm{\phi-\phi_k}_\infty < \epsilon$.
\end{itemize}
We have
\[\PPs{\DP(\mu_n)}{U_{\phi,x,\eta}}\geq \PPs{X\sim\DP(\mu_n)}{\abs{\int_\Omega \phi_k dX - x }<\eta - \epsilon},\quad \forall k>k_0.\]
From Theorem~\ref{thm:ldp_dir}, if $X_n\sim\DP(\mu_n)$, then the sequence
$
\pa{X_n(A_{1,k}),\dots,X_n(A_{n_k,k})}_{n}
$
satisfies an LDP with rate function $I_k$ given by
\[I_k(y_1,\dots,y_{n_k})=\left\{
    \begin{array}{ll}      \KL{\pa{\mu(A_{i,k})}_{i\in [n_k]}}{(y_i)_{i\in [n_k]}} & \mbox{if } (y_i)_{i\in [n_k]} \in  \cM_1\pa{[n_k]}\\
        \infty & \mbox{otherwise.}
    \end{array}
\right.\]
From the Contraction Principle (Fact~\ref{fact:contraction}), we get that $\EEs{X_n}{\phi_k} = \sum_{i\in [n_k]} c_{i,k} X_n(A_{i,k})$  satisfies an LDP with rate function $J_k$ given by
\begin{align*}
    J_k(x)&=\inf \sset{I_k(y_1,\dots,y_{n_k}), 
    (y_i)_{i\in [n_k]} \in  \cM_1\pa{[n_k]},
    \sum_{i\in [n_k]} c_{i,k}y_i = x}
    \\&=
    \inf \sset{\KL{\pa{\mu(A_{i,k})}_{i\in [n_k]}}{\pa{\nu(A_{i,k})}_{i\in [n_k]}}, 
    ~\nu\in \cM_1(\Omega),
    \int_\Omega \phi_k d\nu = x}
    .
\end{align*}
In particular, we obtain the large deviations lower bound
\begin{align*}
&\liminf_n \frac{1}{n}\log\PPs{\DP(\mu_n)}{U_{\phi,x,\eta}}\\&\geq
\liminf_n \frac{1}{n}\log\PPs{X\sim\DP(\mu_n)}{\abs{\int_\Omega \phi_k dX - x }<\eta - \epsilon}\\&\geq -\inf_{\abs{y-x}<\eta-\epsilon} J_k(y)
\\&=
-\inf_{\nu\in \cM_1(\Omega),~\abs{\int_\Omega \phi_k d\nu-x}<\eta-\epsilon} \KL{\pa{\mu(A_{i,k})}_{i\in [n_k]}}{\pa{\nu(A_{i,k})}_{i\in [n_k]}}
\\&\geq 
-\inf_{\nu\in \cM_1(\Omega),~\abs{\int_\Omega \phi d\nu-x}<\eta-2\epsilon} \KL{\pa{\mu(A_{i,k})}_{i\in [n_k]}}{\pa{\nu(A_{i,k})}_{i\in [n_k]}}
\\&\geq 
-\inf_{\nu\in \cM_1(\Omega),~\abs{\int_\Omega \phi d\nu-x}<\eta-2\epsilon} \sup_k\KL{\pa{\mu(A_{i,k})}_{i\in [n_k]}}{\pa{\nu(A_{i,k})}_{i\in [n_k]}}
\\&= 
-\inf_{\nu\in \cM_1(\Omega),~\abs{\int_\Omega \phi d\nu-x}<\eta-2\epsilon} \KL{\mu}{\nu},
\end{align*}
where the last equality is from Fact~\ref{fact:chara_kl}. We can let $\epsilon$ decrease to zero to get
\[\liminf_n \frac{1}{n}\log\PPs{\DP(\mu_n)}{U_{\phi,x,\eta}}\geq -\inf_{\nu\in \cM_1(\Omega),~\abs{\int_\Omega \phi d\nu-x}<\eta} \KL{\mu}{\nu},\]
which is the desired large deviations lower bound for the set $U_{\phi,x,\eta}$. We thus have the large deviations lower bound for a base of the weak topology on
$\cM_1(\Omega)$, and hence for
all open sets. 
\end{proof}

\end{supplement}

\end{document}